\documentclass{amsart}
\usepackage{amssymb,amsthm,amsfonts,amstext}
\usepackage{amsmath}
\usepackage{enumerate}
\numberwithin{equation}{section}
\usepackage{color}
\usepackage[active]{srcltx}

\newtheorem{theorem}{Theorem}[section]
\newtheorem{lemma}[theorem]{Lemma}
\newtheorem{proposition}[theorem]{Proposition}
\newtheorem{corollary}[theorem]{Corollary}
\newtheorem{remark}[theorem]{Remark}
\newtheorem{definition}{Definition}
\newcommand{\mc}[1]{{\mathcal #1}}

\newcommand{\bb}[1]{{\mathbb #1}}

\newcommand{\eps}{\varepsilon}

\newcommand{\<}{\langle}
\renewcommand{\>}{\rangle}
\newcommand{\p}{\partial}
\newcommand{\pfrac}[2]{\genfrac{}{}{}{1}{#1}{#2}}

\newcommand{\at}[2]{\genfrac{}{}{0pt}{}{#1}{#2}}

 \newcommand{\A}{C^{1,2}([0,T]\times [0,1])}
\newcommand{\B}{\mathbb H^\alpha}


\date{}

\begin{document}

\title[Phase transition of a heat equation]{Phase transition of a Heat equation with Robin's boundary conditions and exclusion process}

\author{Tertuliano Franco}
\address{UFBA\\
 Instituto de Matem\'atica, Campus de Ondina, Av. Adhemar de Barros, S/N. CEP 40170-110\\
Salvador, Brasil}
\curraddr{}
\email{tertu@impa.br}
\thanks{}

\author{Patr\'{\i}cia Gon\c{c}alves}
\address{CMAT, Centro de Matem\'atica da Universidade do Minho, Campus de Gualtar, 4710-057, Braga, Portugal}
\curraddr{}
\email{patg@math.uminho.pt}
\thanks{}

\author{Adriana Neumann}
\address{UFRGS, Instituto de Matem\'atica, Campus do Vale, Av. Bento Gon\c calves, 9500. CEP 91509-900, Porto Alegre, Brasil}
\curraddr{}
\email{aneumann@impa.br}
\thanks{}


\begin{abstract}
For a heat equation with Robin's boundary conditions which depends on a parameter $\alpha>0$, we prove that its unique weak solution $\rho^\alpha$ converges, when $\alpha$ goes to zero or to infinity, to the unique weak solution of the heat equation with Neumann's boundary conditions or the heat equation with periodic boundary
conditions, respectively.  To this end, we use uniform  bounds on a Sobolev norm of $\rho^\alpha$ obtained from the hydrodynamic limit of the symmetric slowed exclusion process, plus a careful analysis of boundary terms.
\end{abstract}

\maketitle

\section{Introduction}

 Scaling limits of discrete particle systems is a central question in Statistical Mechanics. In the case of interacting particle systems, where particles evolve according to some rule of interaction, it is of interest to characterize, in the continuum limit, the time trajectory of the spatial density of particles. Such limits are given in terms of solutions of partial differential equations and  different particle systems are governed by different types of partial differential equations, with a large literature on the subject.  As a reference, we cite the book  \cite{kl}.

 In this work we are concerned with convergence of solutions of a particular partial differential equation emerging from particle systems that we describe as follows. Given $\alpha>0$, denote by  $\rho^\alpha$ the unique weak solution of the heat equation with Robin's boundary conditions given by
 \begin{equation*}
\left\{
\begin{array}{ll}
 \partial_t \rho(t,u) \; =\; \Delta \rho(t,u)\,, &t \geq 0,\, u\in (0,1)\,,\\
 \partial_u \rho(t,0) \; =\;\partial_u \rho(t,1)= \alpha(\rho(t,0)-\rho(t,1))\,,  &t \geq 0\,,\\
 \rho(0,u) \;=\; \rho_0(u), &u \in (0,1)\,.
\end{array}
\right.
\end{equation*}
Such equation is related to a particle system with exclusion dynamics, see \cite{fgn,fl}, in a sense which will be precise later. We notice that the boundary conditions of Robin's type as given above mean a passage of mass between $u=0$ and $u=1$. These boundary conditions arise from considering the  particle systems evolving on the discrete torus. Moreover,  they reflect the Fick's Law: the rate at which the mass crosses the boundary is proportional to the difference of densities in each medium.

The main theorem we present here is the following convergence in $L^2$:
\begin{equation*}
\displaystyle \lim_{\alpha\to 0} \rho^\alpha \; = \;  \rho^0\quad \textrm{ and }\quad \displaystyle \lim_{\alpha\to \infty} \rho^\alpha \; = \;  \rho^\infty\;,
\end{equation*}
where $\rho^0$ is the unique weak solution of the heat equation with Neumann's boundary  conditions
 \begin{equation*}
\begin{cases}
 \partial_t \rho(t,u) \; =\; \Delta \rho(t,u)\,,&t \geq 0,\, u\in (0,1)\,,\\
 \partial_u \rho(t,0) \; =\;\partial_u \rho(t,1)= 0\,,&t \geq 0\,,\\
  \rho(0,u) \;=\; \rho_0(u)\,, &u \in (0,1)\,,\\
\end{cases}
\end{equation*}
and  $\rho^\infty$ is the unique weak solution of the heat equation with periodic boundary conditions
\begin{equation*}
\begin{cases}
 \partial_t \rho(t,u) \; =\; \Delta \rho(t,u)\,,\qquad & t \geq 0,\, u\in \bb T\,,\\
\rho(0,u) \;=\; \rho_0(u)\,, &u \in \bb T\,,
\end{cases}
\end{equation*}
 where $\bb T$ above is the continuous torus. The outline of its proof is the following. Based on energy estimates coming from the particle system, we obtain that the set $\{\rho^\alpha\,;\,\alpha>0\}$ is bounded in a Sobolev's type norm, implying its relative compactness. On the other hand, a careful analysis shows that the limit along subsequences of $\rho^\alpha$ are concentrated on weak solutions of the corresponding equations, when $\alpha$ goes to zero or to infinity. Uniqueness of  weak solutions in each case ensures the convergence.

 When $\alpha$ goes to zero or to infinity, the corresponding limits of $\rho^\alpha$ are driven by solutions of partial differential equations of a different kind from the original one. For this reason, we employ the term \emph{phase transition}.  To the best of our knowledge,  this type of result is not a standard one in the partial differential equations literature, without previous nomenclature on it.

 One of the novelties of  this work, besides the aforementioned theorem, is the approach itself: it is used here the framework of probability theory to obtain knowledge on the behavior of  solutions for a class of heat equations with Robin's boundary conditions as given above.

 Next, we describe the particle system which provides the bounding on a Sobolev's type norm for $\rho^\alpha$. This particle system belongs to the class of Markov processes and evolves on the $1$-dimensional discrete torus with $n$ sites. The elements of its state space are called configurations and are such that at each site of the torus, there is at most a particle per site, therefore it coined the name exclusion process.  Its dynamics can be informally described as follows. At each bond of the torus is associated an exponential clock in such a way that clocks associated to different bonds are independent. When a clock rings at a bond, the occupation at the vertices of that bond are interchanged. Of course, if both vertices are occupied or empty, nothing happens. All the clocks have parameter one, except one particular bond, whose parameter is given by $\alpha n^{-\beta}$, where $\alpha,\beta>0$. Or else, this bond slows down the passage of particles across it. It is the existence of this special bond that
gives rise to the boundary conditions of its associated partial differential equation.

For $\beta=1$, the hydrodynamic limit of this exclusion process is a particular case of the processes studied in  \cite{fl}. There, it was proved that the hydrodynamic limit is driven by a generalized partial differential equation involving a Radon-Nikodym derivative with respect to the Lebesgue measure plus a delta of Dirac. As an additional result, we deduce here another proof of this hydrodynamic limit, identifying $\rho^\alpha$ as a solution of a  classical  equation, namely the heat equation with Robin's boundary conditions as given above. Furthermore, by the results proved in \cite{fgn,fl}, we  get that $\rho^\alpha$ has a Sobolev's type norm bounded by a constant that does not depend on $\alpha$. Such constant corresponds to the entropy bound of any measure defined in the state space of the process with respect to its invariant measure.

We point out that, despite knowing the hydrodynamic limit for this process, this different characterization of the limit density of particles, given in terms of a classical partial differential equation, is  new. The most delicate step in the proof of last result is the proof of uniqueness of weak solutions of the heat equation with Robin's boundary conditions, requiring the construction of a inverse of the laplacian operator acting on a suitable domain.

The motivation of this work came from \cite{fgn} where the hydrodynamic limit for the exclusion process with a slow bond was shown to be given by the heat equation with periodic boundary conditions or the heat equation with Neumann's boundary conditions, depending whether $\beta<1$ or $\beta>1$, respectively.  This  suggested to us  that,  when taking the limit in $\alpha$ in the partial differential equation corresponding to $\beta=1$,  one should recover both these equations.

The paper is divided as follows. We give definitions and state our results in Section \ref{s2}. In Section \ref{s3}, we prove uniqueness of weak solutions of the heat equation with Robin's boundary conditions. In Section \ref{s4}, we introduce the exclusion process with a slow bond, we state and  sketch the proof of its hydrodynamic limit and we obtain bounds on a Sobolev's type norm of $\rho^\alpha$. In Section \ref{s5}, we prove our main result, namely the phase transition for the heat equation with Robin's boundary conditions.
In the Appendix, we present some results that are needed in due course.

\textbf{Notations:} We denote by $\bb T$ the continuous one dimensional torus $\bb R/\bb Z$.
By an abuse of notation, we denote $\<\cdot,\cdot\>$ both the inner product in $L^2(\bb T)$ and in $L^2[0,1]$,  and we denote by $\|\cdot\|_{L^2[0,1]}$ its norm. We denote by ${\bf{1}}_A(x)$ the function which is equal to one if $x\in A$ and zero if $x\notin A$ and by $\Delta$ the second space derivative.

\section{Statement of results}\label{s2}
In this section, we begin by defining  weak solutions of the partial differential equations that we deal with, namely the heat equation with periodic, Robin's and Neumann's boundary conditions. In the sequence, we present the exclusion process with a slow bond, we explain its relation with those equations, and how to obtain from it the boundedness of a Sobolev's type norm of weak solutions of the heat equation with Robin's boundary conditions. At last, we state our main result.

\begin{definition}\label{space C^n,m}
For $n,m\in{\mathbb{N}}$ and $A,B$ intervals of $\bb R$ or $\bb T$, let  $C^{n,m}(A\times B)$ be the space of  real valued functions defined on $A\times B$
of class $C^n$ in the first variable and of class $C^m$ in the second variable.
For functions of one variable, we simply write $C^n(A)$.
\end{definition}

Now, we define weak solutions of the partial differential equations that we deal with.
\begin{definition}\label{def edp 1} We say that $\rho$ is a weak solution of the heat equation with periodic boundary conditions
 \begin{equation}   \label{he}
\left\{
\begin{array}{ll}
 \partial_t \rho(t,u) \; =\; \Delta \rho(t,u)\,, \qquad& t \geq 0,\, u\in \bb T\,,\\
  \rho(0,u) \;=\; \rho_0(u), &u \in \bb T\,.
\end{array}
\right.
\end{equation}
if $\rho$ is measurable and, for any  $t\in{[0,T]}$ and any $H\in C^{1,2}([0,T]\times\mathbb{T}) $,
\begin{equation}\label{eqint1}
\begin{split}
&\<\rho_t,\,H_t\>-\<\rho_0,\,H_0\>- \int_0^t\big\< \rho_s,\, \partial_s H_s+\Delta  H_s\big\>\, ds\;=\;0\,.
\end{split}
\end{equation}

\end{definition}

Above and in the sequel, a subindex in a function means a variable, \emph{not a derivative}. For instance, above by $H_s(u)$, we mean $H(s,u)$.

To define a weak solution of the heat equation with Robin's or Neumann's boundary conditions it is necessary to introduce the notion of Sobolev's spaces.

\begin{definition}\label{Sobolevdefinition}
Let  $\mc H^1$ be the set of all locally summable functions $\zeta: [0,1]\to\bb R$ such that
there exists a function $\p_u\zeta\in L^2[0,1]$ satisfying
\begin{equation*}
 \<\partial_uG,\zeta\>\,=\,-\<G,\partial_u\zeta\>\,,
\end{equation*}
for all $G\in C^{\infty}(0,1)$ with compact support.
For $\zeta\in\mc H^1$, we define the norm
\begin{equation*}
 \Vert \zeta\Vert_{\mc H^1}\,:=\, \Big(\Vert \zeta\Vert_{L^2[0,1]}^2+\Vert\partial_u\zeta\Vert_{L^2[0,1]}^2\Big)^{1/2}\,.
\end{equation*}
Let $L^2(0,T;\mc H^1)$ be the space of
 all measurable functions
$\xi:[0,T]\to \mc H^1$ such that
\begin{equation*}
\Vert\xi \Vert_{L^2(0,T;\mc H^1)}^2 \,
:=\,\int_0^T \Vert \xi_t\Vert_{\mc H^1}^2\,dt\,<\,\infty\,.
\end{equation*}
\end{definition}


\begin{definition}\label{heat equation Robin}
We say that $\rho$ is a weak solution of the heat equation with Robin's boundary conditions given by
\begin{equation}\label{her}
\left\{
\begin{array}{ll}
 \partial_t \rho(t,u) \; =\; \Delta \rho(t,u)\,, &t \geq 0,\, u\in (0,1)\,,\\
 \partial_u \rho(t,0) \; =\;\partial_u \rho(t,1)= \alpha(\rho(t,0)-\rho(t,1))\,, \qquad &t \geq 0,\\
 \rho(0,u) \;=\; \rho_0(u), &u \in (0,1)\,.
\end{array}
\right.
\end{equation}
if $\rho$ belongs to $L^2(0,T;\mathcal{H}^1)$ and, for  all $t\in [0,T]$ and for all  $H\in \A$,
\begin{equation}\label{eqint2}
\begin{split}
\< \rho_t,H_t\>\!-\!\<\rho_0,H_0\>
\!-\!\! \int_0^t\!\!\!\big\< \rho_s, \partial_s H_s +&\Delta H_s\big\> ds\!-\!\!\int_0^t\!\!\!(\rho_s(0)\partial_uH_s(0)-\rho_s(1)\partial_uH_s(1))\,ds\\
&+ \int_0^t \alpha(\rho_s(0)-\rho_s(1))(H_s(0)-H_s(1))\,ds=0\,.
\end{split}
\end{equation}
\end{definition}

\begin{definition}\label{heat equation Neumann}
 We say that $\rho$ is a weak solution of the heat equation with Neumann's boundary conditions
 \begin{equation}\label{hen}
\left\{
\begin{array}{ll}
 \partial_t \rho(t,u) \; =\; \Delta \rho(t,u)\,, &t \geq 0,\, u\in (0,1)\,,\\
 \partial_u \rho(t,0) \; =\;\partial_u \rho(t,1)= 0\,, \qquad &t \geq 0\,,\\
 \rho(0,u) \;=\; \rho_0(u), &u \in (0,1)\,.
\end{array}
\right.
\end{equation}
if $\rho$ belongs to $L^2(0,T;\mathcal{H}^1)$ and, for all $t\in [0,T]$ and for all  $H\in \A$,
\begin{equation}\label{eqint3}
\begin{split}
&\< \rho_t,H_t\>\!-\!\<\rho_0,H_0\> \!- \!\!\!\int_0^t\!\!\!\!\big\< \rho_s, \partial_s H_s\!+\!\Delta H_s\big\> ds
\!-\!\!\!\int_0^t\!\!\!(\rho_s(0)\partial_u H_s(0)\!-\!\rho_s(1)\partial_uH_s(1))ds=0.\\
\end{split}
\end{equation}
\end{definition}

Since in Definitions \ref{heat equation Robin} and \ref{heat equation Neumann} we required that
$\rho\in L^2(0,T;\mathcal{H}^1)$,
the integrals at boundary points  are well-defined.  For  more details on Sobolev's spaces, we refer the reader to  \cite{e,l}.

Heuristically, in order to establish the integral equation for the
weak solution of each one of the equations above, one should multiply both sides of the differential equation by a test function $H$, then integrate both in space and time and finally, perform twice a formal integration by parts. Then, applying the  respective boundary conditions we are
lead to the corresponding integral equation. This reasoning also shows  that any strong solution is a weak solution of the respective equation.

We  define a measure $W_\alpha$  on $\bb T$  by
\begin{equation}\label{W}
W_\alpha(du)=du+\frac{1}{\alpha}\,\delta_0(du)\,,
\end{equation}
that is,  $W_\alpha$ is the sum of the Lebesgue measure and the Dirac measure concentrated on $0\in\bb T$ with weight $1/\alpha$.  We denote by $\<\cdot,\cdot\>_\alpha$ the inner product in $L^2$ of $\bb T$  with respect to the measure $W_\alpha$.

\begin{definition}\label{def7}
Let $L^2_{W_{\alpha}}([0,T]\times\bb T)$ be the Hilbert space composed of measurable functions $f:[0,T]\times\bb T\rightarrow{\bb R}$ with
$\|f\|_{\alpha}^2:= \<\!\<f,f\>\!\>_{\alpha}<\infty$, where for $f,g: [0,T] \times \mathbb{T} \to \mathbb{R}$,
\begin{equation*}
\<\!\<f,g\>\!\>_{\alpha}\,=\, \int_0^T   \int_{\bb T}
 f_s( u) \, g_s(u)\,{W_\alpha}(du)\,ds\,.
\end{equation*}
By  $\<\!\<\cdot,\cdot\>\!\>$ we denote the usual inner product corresponding to the Hilbert space $L^2([0,T]\times\bb T)$. Or else,
\begin{equation*}
\<\!\<f,g\>\!\>\,=\, \int_0^T   \int_{\bb T}
 f_s( u) \, g_s(u)\,du\,ds\,.
\end{equation*}
By abuse of notation, we will use the same notation $\<\!\<\cdot,\cdot\>\!\>$ for the inner product on the Hilbert space $L^2([0,T]\times[0,1])$.
\end{definition}

\begin{proposition}\label{uniq_sobolev}
 For any $\alpha>0$, there exists a weak solution $\rho^\alpha:[0,T]\times[0,1]\to [0,1]$ of \eqref{her}. Moreover, such solution is unique and
satisfies the inequality
\begin{equation*}
 \sup_{H}\Big\{\<\!\<\rho^{\alpha},\partial_u H\>\!\>-2\<\!\<H,H\>\!\>_{\alpha}\Big\}\leq K_0\,,
\end{equation*}
where $K_0$ is a constant that does not depend on $\alpha$ and the supremum is taken over functions $H\in C^{\,0,1}([0,T]\times\bb T)$, see Definition \ref{space C^n,m}.
\end{proposition}
The uniqueness  of  weak solutions stated in the proposition above is proved in Section \ref{s3} via the construction of the inverse of the laplacian operator defined on a suitable domain. The existence of a weak solution and the inequality above are proved through the hydrodynamic limit of the symmetric exclusion process with a slow bond,  as shown in Section \ref{s4}.

We state now our main result:
\begin{theorem}\label{pdePT}
 For $\alpha>0$, let $\rho^\alpha:[0,T]\times[0,1]\to [0,1]$ be the unique weak solution of \eqref{her}.
Then,
\begin{equation*}
\displaystyle \lim_{\alpha\to 0} \rho^\alpha \; = \;  \rho^0\quad \textrm{ and }\quad \displaystyle \lim_{\alpha\to \infty} \rho^\alpha \; = \;  \rho^\infty
\end{equation*}
in $L^2([0,T]\times [0,1])$, where
$\rho^0$ and $\rho^\infty$ are the unique weak solutions of equations \eqref{hen} and
\eqref{he}, respectively.
\end{theorem}

\section{Uniqueness of Weak Solutions}\label{s3}
We present here the proof of uniqueness of weak solutions of \eqref{her}.
Since the equation is linear, it is sufficient to consider the initial condition $\rho_0(\cdot)\equiv 0$.

We begin by defining the inverse of the laplacian operator on a suitable domain.
Denote by  $L^2[0,1]^{\perp 1}$ the set of functions $g\in L^2[0,1]$ such that
\begin{equation*}
 \int_{0}^1 g(u)\,du\,=\, 0\,.
\end{equation*}
\begin{definition}\label{operatorext}
Let $\B$ be the space of functions $H:[0,1]\rightarrow\bb{R}$ satisfying
\begin{itemize}
\item  $H\in C^1([0,1])$ and, moreover, the derivative $\partial_u H$ is  absolutely continuous;
\item $\Delta H(u)$ exists Lebesgue almost surely and $ \Delta H\in L^2[0,1]^{\perp 1}$;
\item  $H$ satisfies the boundary conditions
\begin{equation}\label{boundary conditions in Hbc}
\partial_u H(0)=\partial_u H(1)= \alpha(H(0)- H(1))\,.
\end{equation}
\end{itemize}
\end{definition}

In order to obtain the uniqueness of weak solutions, we will construct a inverse of the laplacian operator. However, the laplacian operator is not injective in the domain $\B$.
For this reason, let us define $\B_0$ as set of functions $H\in\B$ such that $H(0)=0$.
\begin{proposition} The operator
$\Delta:\B_0\to L^2[0,1]^{\perp 1}$ is injective.
\end{proposition}
\begin{proof}
Since the operator is linear it is enough to show that its kernel reduces to the null function. For that purpose, let $H\in \B_0$ be such that $\Delta H\equiv 0$ Lebesgue almost surely.  Since $\p_uH$ is absolutely continuous, this implies that $\p_u H$ is constant.  Hence, $H(u)=a+bu$. The unique value of $b$ for which $H$  of this form satisfies \eqref{boundary conditions in Hbc} is $b=0$. Since $H(0)=0$, have that $a=0$, thus $H$ is zero.
\end{proof}

For $g\in L^2[0,1]^{\perp 1}$, define
\begin{equation*}
 [(-\Delta)^{-1}_\alpha g](u)\,:=\,\int_{0}^1 G_\alpha(u,r)\,g(r)\,dr\,,
\end{equation*}
where the function $G_\alpha:[0,1]\times [0,1]\to \bb R$ is given by
\begin{equation*}
 G_\alpha(u,r)\,=\,\frac{\alpha}{\alpha+1}\,u(1-r)-(u-r){\bf 1}_{\{0\leq r\leq u\leq 1\}}\,.
\end{equation*}

\begin{proposition}\label{propinvlaplaciano}
Let $g\in L^2[0,1]^{\perp 1}$. The operator  $(-\Delta)^{-1}_\alpha $ enjoys the following properties:
\begin{enumerate}
\item[(a)]  $(-\Delta)^{-1}_\alpha g\in C^1([0,1])$. Moreover, its first derivative
is absolutely continuous.

\vspace{0.3cm}

\item[(b)] 
\noindent
$\partial_u [(-\Delta)^{-1}_\alpha g](0)=\partial_u [(-\Delta)^{-1}_\alpha g](1)=
\alpha([(-\Delta)^{-1}_\alpha g](0)-[(-\Delta)^{-1}_\alpha g](1))$.

\vspace{0.3cm}

\item[(c)]  $(-\Delta)^{-1}_\alpha g\in \B_0$.

\vspace{0.3cm}

\item[(d)]  $(-\Delta) \big[(-\Delta)^{-1}_\alpha g\big]=g$.

\vspace{0.3cm}

\item[(e)] The operators $(-\Delta):\B_0\to L^2[0,1]^{\perp 1}$ and $(-\Delta)^{-1}_\alpha :L^2[0,1]^{\perp 1}\to \B_0$ are symmetric and non-negative.
\end{enumerate}
\end{proposition}
\begin{proof}
By the definition of $(-\Delta)^{-1}_\alpha$,
\begin{equation*}\label{eeq2.10}
 [(-\Delta)^{-1}_\alpha g](u)\,=\, \frac{\alpha}{\alpha+1}\,u\int_0^1(1-r)g(r)\,dr-u\int_0^u g(r)\,dr+\int_0^u r\, g(r)\,dr\,.
\end{equation*}
By differentiation, we obtain
$$\p_u[(-\Delta)^{-1}_\alpha g](u) =  \frac{\alpha}{\alpha+1}\int_0^1(1-r)g(r)\,dr-\int_0^u g(r)dr\,,$$ implying (a).
Item (b) follows from the assumption $g\in L^2[0,1]^{\perp 1}$.
Items (a) and (b) together imply (c).
By differentiating again the previous equality and recalling (c) we are lead to (d).

It remains to prove (e). Fix $G,H\in \B_0$.  Integration  by parts gives
\begin{equation*}
 \< -\Delta G,H\>\,=\,\<\partial_u G,\partial_u H\>+\partial_u G(0) H(0) -\partial_u G(1) H(1)\,.
\end{equation*}
Since  $G,H\in \B$,
these functions satisfy \eqref{boundary conditions in Hbc}. As a consequence,
\begin{equation*}
\< -\Delta G,H\>\,=\,\<\partial_u G,\partial_u H\>+\,\frac{1}{\alpha}\,\partial_u G(0) \partial_u H(0)\,,
\end{equation*}
which implies symmetry and non-negativity of $\Delta$. The same argument applies for
 $(-\Delta)^{-1}_\alpha $, by item (d).
\end{proof}

\begin{lemma}\label{H2bc}
 Let $\rho$ be a weak solution of \eqref{her}.
Then, for all   $H\in\B$ and for all $t\in [0,T]$ ,
\begin{equation}\label{ext}
\< \rho_t, H\> \,-\, \< \rho_0 , H\>
\,=\, \int_0^t \< \rho_s , \Delta H \>\, ds\,.
\end{equation}
\end{lemma}
\begin{proof}
Fix $H\in \B$.
Let $\{g_n\}_{n\in\bb N}\subset C([0,1])$ be a sequence of functions converging to $\Delta H$ in $L^2[0,1]$
such that $\int_0^1 g_n(u)\,du=0$ for all $n\in\bb N$. Notice that this is possible because $\Delta H$ has zero mean.

Define
 \begin{equation*}
 G_n(u):=H(0)+\partial_uH(0)\, u+ \int_0^u  \int_0^v g_n(r) \, dr\,dv\,.
\end{equation*}
Since $g_n$ has zero mean, then $\partial_uG_n(0)=\partial_uG_n(1)=\p_u H(0)=\p_u H(1)$.
Besides that, $G_n(0)=H(0)$.
It is easy to verify that $G_n\in \A$, see the Definition \ref{space C^n,m}.
Since $\rho(\cdot)$ is a weak solution of equation \eqref{her} and $G_n\in \A$, then we get that
\begin{equation}\label{g_n}
\begin{split}
\< \rho_t, G_n\> \,-\, \< \rho_0, G_n\>
\,=&\, \int_0^t \< \rho_s , g_n \>\, ds+\int_0^t(\rho_s(0)-\rho_s(1))\partial_uH(1)\,ds\\
+&\int_0^t\alpha(\rho_s(0)-\rho_s(1))(H(0)-G_n(1)))\,ds\,.
\end{split}
\end{equation}
We want to take the limit $n\to\infty$ in the previous equation. To this end, notice that $G_n(1)$ converges to
\begin{equation*}
 H(0)+\partial_u H(0)+ \int_0^1  \int_0^v \Delta H(r) \, dr\,dv\,.
\end{equation*}
Since $\Delta H$ is absolutely continuous, the Fundamental Theorem of Calculus can be applied twice, showing that the previous expression is equal to $H(1)$.
Therefore, taking the limit $n\to \infty$ in \eqref{g_n}, we obtain that
\begin{equation*}
\begin{split}
\< \rho_t, H\> \,-\, \< \rho_0, H\>
\,=&\, \int_0^t \< \rho_s ,\Delta H\>\, ds+\int_0^t(\rho_s(0)-\rho_s(1))\partial_uH(1)\,ds\\
-&\int_0^t\alpha(\rho_s(0)-\rho_s(1))(H(0)-H(1))\,ds\,.
\end{split}
\end{equation*}
By \eqref{boundary conditions in Hbc} the last two integral terms on the right hand side of the previous expression cancel, which ends the proof.
\end{proof}

\begin{proposition}\label{prop243}
 Let $\rho$   be a weak solution of \eqref{her}  with $\rho_0(\cdot)\equiv{0}$.
Then, for all $t\in[0,T]$, it holds that
\begin{equation}\label{e2.9}
\big\< \rho_t,(-\Delta)^{-1}_\alpha\rho_t\big\>=-2\int_0^t\<\rho_s,\rho_s\>\,ds\,.\\
\end{equation}
In particular, since equation \eqref{her} is linear,
there exists at most one weak solution with initial condition $\rho_0(\cdot)$.
\end{proposition}
\begin{proof}

We  first claim that $\<\rho_t,1\>=0$ for any time $t\in [0,T]$, if $\rho_0(\cdot)\equiv{0}$. This is a consequence of taking a function $H\equiv 1$ in the integral equation \eqref{eqint2}. Since $\rho$ is bounded, we have also that $\rho\in L^2[0,1]^{\perp 1}$. Or else, the function $\rho$ is in the domain of the operator $(-\Delta)^{-1}_\alpha$.

Take a partition $0=t_0<t_1<\cdots<t_n=t$ of the interval $[0,t]$. Writing a telescopic sum, we get to
\begin{equation*}
\begin{split}
\<\rho_t,(-\Delta)^{-1}_\alpha\rho_t\>-\<\rho_0,(-\Delta)^{-1}_\alpha\rho_0\> =&  \sum_{k=0}^{n-1} \< \rho_{t_{k+1}},(-\Delta)^{-1}_\alpha\rho_{t_{k+1}}\>-\< \rho_{t_k},(-\Delta)^{-1}_\alpha\rho_{t_k}\>\,.
\end{split}
\end{equation*}
 By summing and subtracting the term $\< \rho_{t_{k+1}},(-\Delta)^{-1}_\alpha\rho_{t_k}\>$ for each $k$, the right hand side of the previous expression can be rewritten as
\begin{equation}\label{eq1}
\begin{split}
&\sum_{k=0}^{n-1} \< \rho_{t_{k+1}},(-\Delta)^{-1}_\alpha\rho_{t_{k+1}}\>-\< \rho_{t_{k+1}},(-\Delta)^{-1}_\alpha\rho_{t_k}\>\\
+&\sum_{k=0}^{n-1}\< \rho_{t_{k+1}},(-\Delta)^{-1}_\alpha\rho_{t_k}\>-\<\rho_{t_k},(-\Delta)^{-1}_\alpha\rho_{t_k}\>\,.
\end{split}
\end{equation}
We begin by estimating the second sum above. The first one can be estimated in a similar way because  $(-\Delta)^{-1}_\alpha $ is a symmetric operator.

From   item (c) of Proposition \ref{propinvlaplaciano} and Lemma \ref{H2bc} we get that
\begin{equation}\label{eq2.10}
\begin{split}
 &\< \rho_{t_{k+1}},(-\Delta)^{-1}_\alpha\rho_{t_{k}}\>-\< \rho_{t_{k}},(-\Delta)^{-1}_\alpha\rho_{t_k}\>= -\!\int_{t_k}^{t_{k+1}}\!\! \<\rho_s,\rho_s\>\,ds+\!\int_{t_k}^{t_{k+1}}\!\!\<\rho_s,\rho_s-\rho_{t_k}\>\,ds.\\
\end{split}
\end{equation}
The sum over $k$ of the first integral on the right side of last equality is exactly $-\int_{0}^t\<\rho_s,\rho_s\>ds$.

 We claim now that the sum in $k$ of  the last integral on the right hand of the expression above goes to zero as the mesh of the partition goes to zero.
 To this end, we approximate $\rho$ by a smooth function vanishing simultaneously in a neighborhood of $0$ and $1$.

Let $\iota_\delta:\bb R\to \bb R$ be a smooth approximation of the identity. We extend $\rho_s(\cdot)$ as being  zero outside of the interval $[0,1]$. It is classical that
the convolution $\rho_s*\iota_\delta$ is smooth and converges to $\rho_s(\cdot)$ in $L^2[0,1]$ as $\delta\to 0$.
Let $\Phi_\delta:[0,1]\to \bb R$ a smooth function bounded by one,
equal to zero in  $[0,\delta)\cup (1-\delta, 1]$ and equal to one in $(2\delta,1-2\delta)$.
Define  $$\rho^\delta_s(u)\,=\,(\rho_s*\iota_\delta)(u)\,\Phi_\delta(u)\,.$$
Then, $\rho^\delta_s(\cdot)$ converges to $\rho_s(\cdot)$ in $L^2[0,1]$.  Furthermore, since $\rho^\delta_s(\cdot)$ is smooth and vanishes near $0$ and $1$, it is simple to verify that $\rho^\delta_s\in \B_0$.

Adding and subtracting $\rho^\delta$,  the second integral on the right hand side of equality \eqref{eq2.10} can be written as
\begin{equation*}
\int_{t_k}^{t_{k+1}}\<\rho_s-\rho^\delta_s,\rho_s-\rho_{t_k}\>\,ds+ \int_{t_k}^{t_{k+1}}\<\rho^\delta_s,\rho_s-\rho_{t_k}\>\,ds\,.
\end{equation*}
 Fix $\eps>0$. Since $\rho^\delta$ approximates $\rho$, the Dominated Convergence Theorem gives us that the absolute value of the sum in $k$ of the first integral in the expression above is bounded in modulus by $\eps$ for some $\delta(\eps)$ small.
Take  now $\delta=\delta(\eps)$.
Since $\rho^\delta_s\in \B_0$, applying Lemma \ref{H2bc} we get that the second integral above is equal to
\begin{equation*}
 \int_{t_k}^{t_{k+1}}\int_{t_k}^s \<\rho_r,\Delta \rho^\delta_s\>\,dr\,ds\,,
\end{equation*}
whose absolute value is bounded from above by $C(\rho,\delta)(t_{k+1}-t_k)^2$. This is enough to conclude the proof of \eqref{e2.9}.

Let us  prove now the uniqueness of weak solutions. We notice that as above, we take $\rho_0(\cdot)\equiv{0}$ and therefore we want to prove that $\rho_t(\cdot)\equiv{0}$. Since $\rho_t\in{\bb{L}^2[0,1]^{\perp}}$,  by item (e) of Proposition \ref{propinvlaplaciano}, we have that $\< \rho_t,(-\Delta)^{-1}_\alpha\rho_t\>\geq 0$, for all $t\in[0,T]$.
From \eqref{e2.9} and Gronwall's inequality, we conclude that $\< \rho_t,(-\Delta)^{-1}_\alpha\rho_t\>= 0$, for all $t\in[0,T]$.
From item (d), fixed  $t\in[0,T]$, there exists $f_t\in\B_0$ such that $\rho_t= (-\Delta)f_t$. Hence,
\begin{equation*}
\< \rho_t,(-\Delta)^{-1}_\alpha\rho_t\>=\< -\Delta f_t,f_t\>\,=\,\<\partial_u f_t,\partial_u f_t\> \,.
\end{equation*}
Thus, for all $t\in[0,T]$, $\partial_u f_t(\cdot)=0$ Lebesgue almost surely.
Coming back to $\rho_t= (-\Delta)f_t$ we get that $\rho(\cdot)$ is equal to zero.
This concludes the proof.

\end{proof}

\section{Hydrodynamics and energy estimates}\label{s4}

In this section we introduce a  particle system whose scaling limits are driven by the partial differential equations introduced above.
We first describe the model, then we state the hydrodynamics result and finally we obtain energy estimates which are crucial for the proof of
Proposition \ref{uniq_sobolev}.

\subsection{Symmetric slowed exclusion}

\quad
\vspace{0.2cm}

The symmetric exclusion process with a slow bond is a Markov process  $\{\eta_t:\, t\geq{0}\}$ evolving on $\Omega:=\{0,1\}^{\bb T_n}$, where $\bb T_n=\bb Z/n\bb Z$ is the one-dimensional discrete torus with $n$ points. It is characterized via its infinitesimal generator $\mathcal{L}_{n}$ which  acts on functions $f:\Omega\rightarrow \bb{R}$ as
\begin{equation*}
\mathcal{L}_{n}f(\eta)=\sum_{x\in \bb T_n}\,\xi^{n}_{x,x+1}\,\big[f(\eta^{x,x+1})-f(\eta)\big]\,,
\end{equation*}
  being the rates given by \begin{equation*}
\xi^{n}_{x,x+1}\;=\;\left\{\begin{array}{cl}
\alpha n^{-\beta}, &  \mbox{if}\,\,\,\,x=-1\,,\\
1, &\mbox{otherwise\,,}
\end{array}
\right.
\end{equation*}
and $\eta^{x,x+1}$ is the configuration obtained from $\eta$ by exchanging the variables $\eta(x)$ and $\eta(x+1)$, namely
\begin{equation*}
\eta^{x,x+1}(y)=\left\{\begin{array}{cl}
\eta(x+1),& \mbox{if}\,\,\, y=x\,,\\
\eta(x),& \mbox{if} \,\,\,y=x+1\,,\\
\eta(y),& \mbox{otherwise}\,.
\end{array}
\right.
\end{equation*}

%

The dynamics of this process can be informally described as follows. At each bond $\{x,x+1\}$, there is an exponential clock
of parameter $\xi^{n}_{x,x+1}$. When this clock rings, the value of $\eta$ at the vertices of this bond are exchanged. This means that particles can cross all the bonds at rate $1$, except the bond $\{-1,0\}$, whose dynamics is slowed down as $\alpha n^{-\beta}$, with  $\alpha>0$ and $\beta\in{[0,\infty]}$.  It is understood here that $n^{-\infty}=0$ and $\infty\cdot 0=0$.

 It is well known that the Bernoulli product measures on $\Omega$ with parameter $\gamma\in{[0,1]}$, denoted by
$\{\nu^n_\gamma : 0\le \gamma \le 1\}$, are invariant for the dynamics   introduced above.  This means that if $\eta_0$ is distributed according to $\nu^n_\gamma$, then $\eta_t$ is also distributed according to $\nu^n_\gamma$ for any $t>0$. Moreover,  the measures $\{\nu^n_\gamma : 0\le \gamma \le 1\}$ are also
reversible.

In order to keep notation simple, we write $\eta_t:=\eta_{tn^2}$ so that $\{\eta_t: t\ge 0\}$ turns out to be the Markov process on $\Omega$
associated to the generator $\mathcal{L}_n$ speeded up by
$n^2$. We notice that we do not index the process neither in $\beta$ nor in $\alpha$.

 The trajectories of $\{\eta_t : t\ge 0\}$ live on the space $\mc D(\bb R_+, \Omega)$, i.e., the path space of
c\`adl\`ag trajectories with values in $\Omega$. For a
measure $\mu_n$ on $\Omega$, we denote by $\bb P^{\alpha,\beta}_{\mu_n}$ the
probability measure on $\mc D(\bb R_+, \Omega)$ induced by $\mu_n$ and $\{\eta_t : t\ge 0\}$ and we denote by  $\bb E_{\mu_n}^{\alpha,\beta}$
expectation with respect to $\bb P^{\alpha,\beta}_{\mu_n}$.

\subsection{Hydrodynamical phase transition}\label{hid}
\quad
\vspace{0.2cm}

{In order to state the hydrodynamical limit  we introduce the empirical measure process as follows.
We denote by $\mc M$ the space of positive measures on $\bb T$ with total
mass bounded by one, endowed with the weak topology. For $\eta\in{\Omega}$, let
$\pi^{n}(\eta, \cdot) \in \mc M$ be given by:
\begin{equation*}
\pi^{n}(\eta,du) \;=\; \pfrac{1}{n} \sum _{x\in \bb T_n} \eta (x)\,
\delta_{x/n}(du)\,,
\end{equation*}
where $\delta_y$ is the Dirac measure concentrated on $y\in \bb T$. For $t\in{[0,T]}$, let $\pi^{n}_t(\eta,du):=\pi^n(\eta_t,du)$.}

For a test function $H:\bb T \to \bb R$ we use the following notation
\begin{equation*}
\<\pi^n_t, H\>:=\int H(u)\pi_t^n(\eta,du)\;=\; \pfrac 1n \sum_{x\in\bb T_n}
H (\pfrac{x}{n})\, \eta_t(x)\,.
\end{equation*}
We use this notation since for $\pi_t$ absolutely continuous with respect to the Lebesgue measure with density
$\rho_t$, we write $\<\rho_t, H\>$
for $\<\pi_t, H\>$.  

{Fix $T>0$. Let $\mc D([0,T], \mc M)$ be the space of
c\`adl\`ag trajectories with values in  $\mc M$ and endowed with the
\emph{Skorohod} topology.  For each probability measure $\mu_n$ on
$\Omega$, denote by $\bb Q^{\alpha,\beta}_{n,\mu_n}$ the measure on
the path space $\mc D([0,T], \mc M)$ induced by $\mu_n$ and
the empirical process $\pi^n_t$ introduced above.}

In order to state our first result related to the hydrodynamics of this model, we need to impose some conditions on the initial distribution of the process.

\begin{definition} \label{def associated measures}
A sequence of probability measures $\{\mu_n\}_{n\in\bb N}$ on $\Omega$ is
said to be associated to a profile $\rho_0 :\bb T \to [0,1]$ if, for every $\delta>0$ and every $H\in C(\bb T)$,
\begin{equation}\label{associated}
\lim_{n\to\infty}
\mu_n \Big[ \eta:\, \Big\vert \pfrac 1n \sum_{x\in\bb T_n} H(\pfrac{x}{n})\, \eta(x)
- \int_{\bb{T}} H(u)\, \rho_0(u) du \Big\vert > \delta \Big]\;=\; 0\,.
\end{equation}
\end{definition}

Now, we state the dynamical phase transition at the hydrodynamics level for the slowed exclusion process introduce above. We notice that this result is an improvement of the main theorem of \cite{fgn}, { since we are able to identify the hydrodynamic equation for $\beta=1$ as being the heat equation with Robin's boundary conditions as given in \eqref{her}.}

\begin{theorem} \label{th:hlrm}
Fix $\beta\in [0,\infty]$ and $\rho_0: \mathbb{T} \to [0,1]$ continuous by parts. Let $\{\mu_n\}_{n\in\bb N}$ be
a sequence of probability measures  on $\Omega$ associated to $\rho_0(\cdot)$. Then, for any $t\in [0,T]$, for every $\delta>0$ and every $H\in C(\mathbb{T})$:
\begin{equation*}
\lim_{n\to\infty}
\mathbb{P}_{\mu_n}^{\alpha,\beta} \Big[\eta_. : \, \Big\vert \pfrac{1}{n} \sum_{x\in\mathbb{T}_n}
H\big(\pfrac{x}{n}\big)\, \eta_t(x) - \int_{\bb T}H(u)\rho(t,u)du \Big\vert
> \delta \Big] \;=\; 0\,,
\end{equation*}
 where:
\begin{itemize}
\item
if $\beta\in[0,1)$, $\rho(t,\cdot)$ is the unique weak solution of \eqref{he};

\vspace{0.1cm}

\item
if $\beta=1$, $\rho(t,\cdot)$ is the unique weak solution of \eqref{her};

\vspace{0.1cm}

\item
 if $\beta\in(1,\infty]$, $\rho(t,\cdot)$ is the unique weak solution of  \eqref{hen}.
\end{itemize}

\end{theorem}

\begin{proof}
The proof of last result is given in \cite{fgn} for $\beta\in{[0,1)}$ and $\beta\in(1,\infty)$. We also notice that for $\beta=\infty$, the same arguments as used in \cite{fgn} for $\beta\in(1,\infty)$ fit the case $\beta=\infty$ and for that reason we also omit the proof in this case.

Finally, for $\beta=1$, the proof of the hydrodynamic limit can be almost all adapted from the strategy of \cite{fgn} and is the usual in stochastic process: tightness, which means relative compactness, plus uniqueness of limit points. We recall that the proof of tightness is very similar to the one given in \cite{fgn} and for that reason we omitted it. Nevertheless, the characterization of limit points is essentially different from \cite{fgn},  since here we identify the solutions as weak solutions of the heat equation with Robin's boundary conditions given in \eqref{her}. We proceed by presenting the proof of last statement.

Recall the definition of $\{\bb Q^{\alpha,\beta}_{n,\mu_n}\}_{n\in \bb N}$. In order to keep notation simple and since $\beta=1$ we do not index these measures nor $\mathbb{P}_{\mu_n}^{\alpha, \beta}$ on $\beta$. Let $\bb Q_*^{\alpha}$ be a limit point of $\{\bb Q^{\alpha}_{n,\mu_n}\}_{n\in \bb N}$ whose existence is a consequence of Proposition 4.1 of \cite{fgn} and assume, without loss of generality, that $\{\bb Q^{\alpha}_{n,\mu_n}\}_{n\in \bb N}$
converges to $\bb Q_*^{\alpha}$, as $n\to \infty$. Now, we prove that  $\bb Q_*^{\alpha}$ is concentrated on trajectories of measures absolutely
continuous with respect to the Lebesgue measure: $\pi(t,du) = \rho(t,u) du$, whose density
$\rho(t,u)$ is  the unique  weak solution of \eqref{her}.

At first we notice that by Proposition 5.6 of \cite{fgn}, $\bb Q_*^{\alpha}$ is concentrated on trajectories absolutely continuous with
respect to the Lebesgue measure $\pi_t(du)=\rho(t,u)\,du$ such that,
$\rho(t,\cdot)$ belongs to $L^2(0,T;\mc H^1)$.
It is well known that the Sobolev space $\mc H^1$ has
special properties: all its elements are
 absolutely continuous functions  with bounded variation, see \cite{e},
therefore with  well defined lateral limits. Such property is inherited by
 $L^2\big(0,T;\mc H^1)$ in the sense that we can
integrate in time the lateral limits.

Let $H\in\A$. We begin by claiming  that
\begin{equation*}
\begin{split}
\bb Q^\alpha_* \Bigg[\pi_\cdot:\,  \<\rho_t, & H_t \> -  \<\rho_0, H_0 \> -\int_0^t
\big\<\rho_s , \partial_sH_s +\Delta H_s\big\> \,ds \\
&-\,\int_0^t\Big(\rho_s(0)\partial_u H_s(0)-\rho_s(1)\partial_u H_s(1) \Big)\,ds \\
&+\,\int_0^t\alpha\Big(\rho_s(0)-\rho_s(1)\Big)\Big(H_s(0)-H_s(1)\Big)\,ds \,=\,0,\quad\forall t\in[0,T]\, \Bigg]\,=\,1\,.
\end{split}
\end{equation*}
In order to prove last equality, its enough to show that, for every $\delta >0$,
\begin{equation*}
\begin{split}
\bb Q^\alpha_* \Bigg[\pi_\cdot:\sup_{0\leq t\leq T}\,\Bigg\vert\,   \<\rho_t,& H_t \> \,-\,  \<\rho_0, H_0 \> \,-\, \int_0^t  \,
\big\<\rho_s , \partial_sH_s +\Delta H_s\big\> \,ds \\
&-\,\int_0^t\Big(\rho_s(0)\partial_u H_s(0)-\rho_s(1)\partial_u H_s(1) \Big)\,ds \\
&+\,\int_0^t\alpha\Big(\rho_s(0)-\rho_s(1)\Big)\Big(H_s(0)-H_s(1)\Big)\,ds \,\Bigg\vert\,>\,\delta\, \Bigg]=0\,.
\end{split}
\end{equation*}
Since the boundary integrals are not well-defined in $\mc D\big([0,T],\mc M\big)$, we cannot use
directly Portmanteau's Theorem. To avoid this technical obstacle, fix $\eps>0$ and let $\iota_\eps(u)=\pfrac{1}{\eps}\,\textbf 1_{(0,\eps)}(u)$ and $\tilde \iota_\eps(u)=\pfrac{1}{\eps}\,\textbf 1_{(1-\eps,1)}(u)$
be approximations of the identity in the continuous torus.
Now, adding and subtracting the convolution of $\rho(t,u)$ with
$\iota_\eps$  and $\tilde \iota_\eps$, we can bound from above the previous probability by the sum of
\begin{equation*}\label{prob 1sim}
\begin{split}
\bb Q^\alpha_* \Bigg[\pi_\cdot:\sup_{0\leq t\leq T}\,\Bigg\vert\, &  \<\rho_t, H_t \> \,-\,  \<\rho_0, H_0 \> \,-\, \int_0^t  \,
\big\<\rho_s , \partial_sH_s +\Delta H_s\big\> \,ds \\
&-\,\int_0^t\Big((\rho_s*\iota_\eps)(0)\partial_u H_s(0)-(\rho_s*\tilde\iota_\eps)(1)\partial_u H_s(1) \Big)\,ds \\
&+\,\int_0^t\alpha\Big((\rho_s*\iota_\eps)(0)-(\rho_s*\tilde\iota_\eps)(1)\Big)\Big(H_s(0)-H_s(1)\Big)\,ds \,\Bigg\vert\,>\,\delta/3\, \Bigg]\,,
\end{split}
\end{equation*}
with the probability of two sets, each one of them decreasing as $\eps\to 0$, to sets of null probability as a consequence of convolutions being  suitable  averages of $\rho$ around the boundary points $0$ and $1$.
Now, we claim that we can use  Portmanteau's Theorem and Proposition A.3 of \cite{fgn} in order to conclude that the previous probability is bounded from above by
\begin{equation*}
\begin{split}
 \varliminf_{n\to \infty}\bb Q^{\alpha}_{n,\mu_n} \Bigg[\pi_\cdot:&\sup_{0\leq t\leq T}\,\Bigg\vert\,
   \<\rho_t, H_t \> \,-\,  \<\rho_0, H_0 \> \,-\, \int_0^t  \,
\big\<\rho_s , \partial_sH_s +\Delta H_s\big\> \,ds \\
&-\,\int_0^t\Big((\rho_s*\iota_\eps)(0)\partial_u H_s(0)-(\rho_s*\tilde\iota_\eps)(1)\partial_u H_s(1) \Big)\,ds \\
&+\,\int_0^ta \Big((\rho_s*\iota_\eps)(0)-(\rho_s*\tilde\iota_\eps)(1)\Big)\Big(H_s(0)-H_s(1)\Big)\,ds \,\Bigg\vert\,>\,\delta/3\, \Bigg]\,.
\end{split}
\end{equation*}
Although the functions $H_t$, $H_0$, $\p_sH_s +\Delta H_s$, $\iota_\eps(\cdot,1)$ and $\tilde\iota_\eps(\cdot,0)$ may not belong to $C(\bb T)$,
we can proceed as in Section 6.2 of \cite{fgn} in order to justify the boundedness of the previous expression.
Next we outline the main arguments involved  in that procedure. Firstly, we replace each one of these functions by continuous functions which
 coincide with the original ones in the torus, except on a small neighborhood of their discontinuity points and such that their $L^\infty$-norm is bounded from above by
 the $L^\infty$-norm of the respective original functions.
By the exclusion rule, the set where we compare this change has small probability.
 Thus, in the presence of continuous functions, we  apply Portmanteau's Theorem and Proposition A.3 of \cite{fgn}.
After this, we return back to the original functions using the same arguments.

Recall that we consider $\bb T_n$ embedded in $\bb T$ and notice that $(\pi^n*\iota_\eps)(\frac{0}{n})=\eta^{\eps n}(0)$ and $(\pi^n*\tilde\iota_\eps)(\frac{1}{n})=\tilde \eta^{\eps n}(n-1)$, where
{\begin{equation}
\eta^{\eps n}(0)=\frac{1}{\eps n}\sum_{y=1}^{\lfloor \eps n\rfloor}\eta(y) \hspace{1cm} \textrm{and} \hspace{1cm}\tilde\eta^{\eps n}(n-1)=\frac{1}{\eps n}\sum_{y=\lfloor n-\eps n\rfloor}^{ n-1 }\eta(y),
\end{equation}}
where  $\lfloor u\rfloor$ denotes the biggest integer smaller or equal to $u$.
By  definition of  $\bb Q^{\alpha}_{n,\mu_n}$ and by summing and subtracting the term $\int_{0}^tn^2 \mc L_{n}\<\pi_s^n,\, H_s \>ds$ inside the supremum above,
we can bound the previous probability by
\begin{equation*}
\bb P^{\alpha}_{\mu_n} \Bigg[\eta_\cdot\,:\,\sup_{0\leq t\leq T}\,\Bigg\vert\,
\<\pi^n_t, H_t \> \,-\,  \<\pi^n_0, H_0 \> \,-\, \int_0^t  \,
\<\pi^n_s ,\p_s  H_s \>+n^2 \mc L_{n}\<\pi_s^n,\, H_s \> \,ds \,\Bigg\vert\,>\,\delta/6\, \Bigg]
\end{equation*}
and
\begin{equation*}
\begin{split}
\bb P^{\alpha}_{\mu_n} \Bigg[\eta_\cdot\,:\,\sup_{0\leq t\leq T}\,\Bigg\vert\,&
 \int_0^t  \,  n^2 \mc L_{n}\<\pi_s^n,\, H_s \> \,ds
-\, \int_0^t  \,  \<\pi^n_s ,\Delta H_s \> \,ds \\
&-\,\int_0^t\Big(\eta_s^{\eps n}(0)\partial_u H_s(0)-\tilde\eta_s^{\eps n}(n-1)\partial_u H_s(1) \Big)\,ds \\
&+\,\int_0^t\alpha\Big(\eta_s^{\eps n}(0)-\tilde\eta_s^{\eps n}(n-1)\Big)
\Big(H_s(0)-H_s(1)\Big)\,ds \,\Bigg\vert\,>\,\delta/6\, \Bigg]\,.
\end{split}
\end{equation*}

 { By Dynkin's formula, the expression inside the supremum in  the first probability above, is a martingale  that we denote by $\mc M^{n}_t(H)$.
 A simple computation shows that $\mc M^{n}_t(H)$ converges to zero in $L^2(\bb P^{\alpha}_{\mu_n})$
 as $n\to \infty$, and then, by Doob's inequality, the first probability
vanishes as $n\to \infty$, for every $\delta>0$.}
Now we treat the remaining term. Using the expression for
 $n^2\mc L_{n}\<\pi_s^n,\, H_s \>$,
we can bound the previous probability  by the sum of
\begin{equation*}
 \bb P^{\alpha}_{\mu_n} \Bigg[\eta_\cdot\,:\,\sup_{0\leq t\leq T}\,\Bigg\vert\,\int_0^t  \,  \<\pi^n_s ,\Delta  H_s \> \,ds-  \int_0^t
 \pfrac{1}{n}\sum_{x\neq n-1, 0} \eta_{s}(x)\Delta_nH_s \Big(\frac{x}{n}\Big)\,ds
\,\Bigg\vert\,>\,\delta/18\, \Bigg]\,,
\end{equation*}
\begin{equation*}
\begin{split}
\bb P^{\alpha}_{\mu_n} \Bigg[\eta_\cdot\,:\,\sup_{0\leq t\leq T}\,&\Bigg\vert\,
 \int_0^t\Big(\eta_s^{\eps n}(0)\partial_u H_s(0)-\tilde\eta_s^{\eps n}(n-1)\partial_u H_s(1) \Big)\,ds \\
&-\int_0^t \Big( \eta_{s}(0)\nabla_{\!n} H_s(0)\,-\,\eta_{s}(n-1)\nabla_{\!n} H_s(n-2)\Big)\,ds
\,\Bigg\vert\,>\,\delta/18\, \Bigg]
\end{split}
\end{equation*}
and
\begin{equation*}
\begin{split}
 \bb P^{\alpha}_{\mu_n} \Bigg[\eta_\cdot\,:\,\sup_{0\leq t\leq T}\,\Bigg\vert\,
 \int_0^t\alpha\Big(\eta_s^{\eps n}&(0)-\tilde\eta_s^{\eps n}(n-1)\Big)
\Big(H_s(0)-H_s(1)\Big)\,ds\\
&- \int_0^t\alpha\Big(\eta_s(0)-\eta_s(n-1)\Big)
\nabla_{\!n} H_s(n-1)\,ds \,\Bigg\vert\,>\,\delta/18\, \Bigg],
\end{split}
\end{equation*}
where for $x\in{\mathbb{T}_n}$, $$\Delta_nH\Big(\frac{x}{n}\Big)=n^2\Big(H\Big(\frac{x+1}{n}\Big)+H\Big(\frac{x-1}{n}\Big)-2H\Big(\frac{x}{n}\Big)\Big)$$ is the discrete laplacian and
  $$\nabla_n H\Big(\frac{x}{n}\Big)=n\Big(H\Big(\frac{x+1}{n}\Big)-H\Big(\frac{x}{n}\Big)\Big)$$ is the discrete derivative.
Since $H\in\A$, the discrete laplacian of $H_s$, namely $\Delta_n$, converges uniformly to the continuous laplacian of $H_s$, that is $\Delta H$, as $n\to\infty$, which is enough to conclude that the first probability is null.

To prove that the remaining probabilities are null, we observe that the discrete derivative of $H$, namely $\nabla_{\!n} H_s$,
converges uniformly to the continuous derivative, that is $\partial_u H_s$, as $n\to\infty$ and $\nabla_{\!n} H_s(n-1)$
converges uniformly to $H_s(0)-H_s(1)$, as $n\to\infty$, since $H\in\A$.
By the exclusion constrain and approximating the integrals by riemannian sums, the previous probabilities
 vanish as long as we show that
\begin{equation*}
\begin{split}
 \bb P^{\alpha}_{\mu_n} \Bigg[\eta_\cdot\,:\,\sup_{0\leq t\leq T}\,\Bigg\vert\,
 \int_0^t\Big(\eta_s^{\eps n}(0)&-\eta_{s}(0)\Big)\partial_u H_s(0)\\
&-\Big(\tilde\eta_s^{\eps n}(n-1)-\eta_{s}(n-1)\Big)
\partial_u H_s(1) \,ds \,\Bigg\vert\,>\,\delta\, \Bigg]\\
 \bb P^{\alpha}_{\mu_n} \Bigg[\eta_\cdot\,:\,\sup_{0\leq t\leq T}\,\Bigg\vert\,
 \int_0^t\alpha\Big\{\Big(&\eta_s^{\eps n}(0)-\tilde\eta_s^{\eps n}(n-1)\Big)\\&-\Big(\eta_s(0)-\eta_s(n-1)\Big)\Big\}
\Big(H_s(0)-H_s(1)\Big)\,ds \,\Bigg\vert\,>\,\delta\, \Bigg]\,.
\end{split}
\end{equation*}
converge to zero, as $\eps\to0$, for all $ \delta>0$. This is a consequence of Lemma 5.4 of \cite{fgn}.

\end{proof}

\subsection{Energy estimates}

\quad
\vspace{0.2cm}

The proof of Proposition  \ref{uniq_sobolev} is a consequence of energy estimates obtained by means of the symmetric slowed exclusion process introduced above and it can be summarized as follows.

 Firstly, we notice that the existence of weak solutions of equation \eqref{her}
is granted by tightness proved in \cite{fgn} together with the characterization of the limiting measure $\bb Q^{\alpha}_*$ given above.

Secondly, uniqueness was proved in Section \ref{s3}.

Finally, to prove the last statement we introduce the next proposition which  is usually denominated by \emph{energy estimate}. It says that any limit measure $\bb Q^{\alpha}_*$ is concentrated on functions with finite energy. Moreover, it says that the expected energy is also finite. Such result makes the link between the particle system $\{\eta_t;\,t\geq{0}\}$ and the weak solution of the heat equation with Robin's boundary condition given in \eqref{her}.

\begin{proposition}\label{Prop_03}
Let $\bb Q_*^{\alpha}$ be a limit point of $\{\bb Q^{\alpha}_{n,\mu_n}\}_{n\in \bb N}$. Then,
\begin{equation*}
\bb E_{\bb Q_*^{\alpha}} \Big[  \sup_{H} \Big\{ \<\!\<\rho,  \,  \partial_u H\>\!\>
\,- \, 2 \<\!\<H ,\,H\>\!\>_{\alpha}\Big\} \Big] \, \le \,K_0\,,
\end{equation*}
where $K_0$ is a constant that does not depend on $\alpha$ and the supremum is taken over functions $H\in C^{\,0,1}([0,T]\times\bb T)$, see Definition \ref{space C^n,m}.
\end{proposition}
Since the proof of this proposition follows the same lines of \cite[Subsection 5.2]{fgn}, it is omitted.\medskip

As seen in this section, the measure $\bb Q^{\alpha}_*$ is
concentrated in  weak solutions of the heat equation with Robin's boundary conditions as given in \eqref{her}.
In Section \ref{s3}, it was proved uniqueness of such weak solutions. This implies that the measure $\bb Q^{\alpha}_*$ is, in fact, a delta of Dirac concentrated on the unique weak solution of \eqref{her}. Denote this solution by $\rho^\alpha$.
By the previous proposition, we conclude that
\begin{equation*}
 \sup_{H} \Big\{ \<\!\<\rho^\alpha,  \,  \partial_u H\>\!\>
\,- \, 2 \<\!\<H ,\,H\>\!\>_{\alpha}\Big\} \le \,K_0\,.
\end{equation*}
where $K_0$ is a constant that does not depend on $\alpha$ and the supremum is taken over functions $H\in C^{\,0,1}([0,T]\times\bb T)$. This proves  Proposition
 \ref{uniq_sobolev}.

\section{Proof of Theorem \ref{pdePT}.}\label{s5}
 Since we  proved Proposition \ref{uniq_sobolev}, from now on,  for fixed $\alpha>0$, we denote
by $\rho^\alpha:[0,T]\times[0,1]\to [0,1]$ the unique weak solution of \eqref{her}. We notice that $\rho^\alpha$ takes values between 0 and 1 since we imposed the same condition for $\rho_0$.

Our scheme of proof has the following steps:

In  Proposition \ref{propositL1}, we prove that the set $\{\rho^{\alpha}:\,\alpha>0\}$
is bounded in $L^2(0,T;\mc H^1)$. This  guarantees  the relative compactness of this set.
  In Proposition \ref{propositL2}, we prove that any limit of a convergent
 subsequence of  $\{\rho^{\alpha_n}\}_{n\in\bb N}$ is in $L^2(0,T;\mc H^1)$.
In Proposition \ref{novolema}, we obtain some smoothness  of $\rho^\alpha$ on time, that we need in order to take limits in $\alpha$.

The next step will be to analyze separately each term of  integral equation \eqref{eqint2}
 to obtain asymptotic results in $\alpha$ for its terms.
 Proposition \ref{lemmaintegrais0} and Proposition \ref{lemmaintegrais2} cover the limit of terms in the integral equation \eqref{eqint2} that can be treated in the same way both for $\alpha\to 0$ and $\alpha\to\infty$.  Proposition \ref{lemmapdePT1} and Proposition \ref{lemmapdePT2} cover the limit of a integral term for the case $\alpha\to 0$ and $\alpha\to \infty$, respectively.

 These convergences of the integral terms will show that any convergent subsequence of $\{\rho^\alpha:\,\alpha>0\}$ with $\alpha\to 0$ converges to the unique weak solution of the heat equation with Neumann's boundary conditions, and any convergent subsequence of $\{\rho^\alpha:\,\alpha>0\}$ with $\alpha\to \infty$ converges to the unique weak solution of the heat equation with periodic boundary conditions. Putting this together with the relative compactness of the set  $\{\rho^\alpha:\,\alpha>0\}$ the convergence follows.

 We introduce first a space of test functions that will be used in the sequel.

\begin{definition}\label{def9}
The space $C_c$ consists of functions $H\in C^{\,0,1}([0,T]\times [0,1])$
 with compact support in $[0,T]\times(0,1)$.
\end{definition}

\begin{proposition}\label{propositL1}
 The set $\{\rho^\alpha:\,\alpha>0\}$ is bounded in $L^2(0,T;\mc H^1)$.
\end{proposition}
\begin{proof}
We begin by observing that Proposition \ref{uniq_sobolev} implies the inequality
\begin{equation}\label{hip1}
 \<\!\<\rho^{\alpha},\partial_u H\>\!\>-2\<\!\<H,H\>\!\>\leq K_0\,,
\end{equation}
for all $H\in C_c$.   This is a consequence of the simple fact that, if $H$ vanishes at a neighborhood of $0$ and $1$, then $\<\!\<H,H\>\!\>=\<\!\<H,H\>\!\>_{\alpha}$, for all $\alpha>0$.

An application of the  Riesz Representation Theorem gives us that
\begin{equation*}
\sup_{H\in C_c}\,\Big\{ \<\!\<\rho^\alpha, \p_u H\>\!\>- 2 \<\!\< H,H\>\!\>\Big\}\,=\,\pfrac{1}{8}\int_0^T\Vert\p_u\rho^\alpha_t\Vert^2\,dt\,.
\end{equation*}
Aiming to concentrate in the main facts, the proof of  previous equality  is postponed to Corollary \ref{lema3} of the Appendix.

 Provided by the inequality above  and recalling that $K_0$ does not depend on $\alpha$,  we conclude
that $\{\rho^\alpha:\,\alpha>0\}$ is bounded in $L^2(0,T;\mc H^1)$.
 \end{proof}

The boundedness of $\{\rho^\alpha:\,\alpha>0\}$ in $L^2(0,T;\mc H^1)$ implies a compact embedding of $\{\rho^\alpha:\,\alpha>0\}$  in  $L^2([0,T]\times [0,1])$. This is a particular case of the Rellich-Kondrachov's
 Theorem for spaces involving time, that can be found in \cite{TE}.
To verify it in detail, we list the exact steps: following the notation of \cite[Page 271, Subsection 2.2]{TE},
take $X_0=X=\mc H^1$, $X_1=L^2[0,1]$ and notice that any Hilbert space is reflexive. This attains the hypothesis of
\cite[Theorem 2.1, page 271]{TE} and corresponds to the case we consider. By this compact embedding,
 any sequence $\{\rho^{\alpha_n}\}_{n\in\bb N}$ has a convergent subsequence
 in $L^2([0,T]\times [0,1])$.\medskip

 Next, we show that the limit of a convergent subsequence of $\{\rho^{\alpha}:\,\alpha>0\}$ is
 in the space $L^2(0,T;\mc H^1)$.

\begin{proposition}\label{propositL2}
If  $\rho^*$  is the limit in $L^2([0,T]\times [0,1])$  of some sequence in the set $\{\rho^{\alpha}:\,\alpha>0\}$, then $\rho^* \in L^2(0,T;\mc H^1)$.
\end{proposition}
\begin{proof}
Suppose that $\rho^{\alpha_n}$ converges to $\rho^*$ in $L^2([0,T]\times [0,1])$,  as $n\to\infty$. By Proposition \ref{uniq_sobolev},
for each $n\in\bb N$, $\rho^{\alpha_n}$ satisfies  \eqref{hip1} for any $H\in C_c$:
 \begin{equation*}
 \int_0^T\<\rho^{\alpha_n}_s,\,\partial_u H_s\,\>\,ds-2\int_0^T\<H_s,\,H_s\,\>\,ds\leq K_0
\end{equation*}
and $K_0$ does not depend neither on $n$ nor $H$.
Taking the limit  $n\to \infty$ in the previous inequality, we get that
\begin{equation*}
 \int_0^T\<\rho^{*}_s,\,\partial_u H_s\,\>\,ds-2\int_0^T\<H_s,\,H_s\,\>\,ds\leq K_0\,.
\end{equation*}
Replacing $H$ by $yH$ in the previous inequality and then minimizing over $y\in \bb R$ gives that
\begin{equation*}
\begin{split}
\varphi:\;&C_c\to{\bb  R}\\
&H\mapsto  \int_0^T\<\rho^{*}_s,\,\partial_u H_s\,\>\,ds
 \end{split}
 \end{equation*}
 is a bounded linear functional. Notice that the set $C_c$ is  dense in $L^2([0,T]\times [0,1])$. Hence, by the Riesz Representation Theorem, there exists $\partial_u\rho^{*}\in L^2([0,T]\times [0,1])$ such that
\begin{equation*}
 \int_0^T\<\rho^{*}_s,\,\partial_u H_s\,\>\,ds=- \int_0^T\<\partial_u\rho^{*}_s,\, H_s\,\>\,ds\,,
\end{equation*}
for all functions $H\in C_c$, which is the same as saying that is $\rho^*$ belongs to $L^2(0,T;\mc H^1)$.
\end{proof}

Now, we analyze the integral equation \eqref{eqint2}.  Integrating  by parts, it  can be rewritten as
\begin{equation}\label{eqintbyparts}
\begin{split}
 \<\rho^\alpha_t,\,H_t\,\> -\<\rho^\alpha_0,\,H_0\,\>+\int_0^t\<\, \partial_u\rho^\alpha_s& , \,\partial_u  H_s\,\>\, ds-
\int_0^t\< \rho^\alpha_s, \, \partial_s H_s\,\>\, ds\\
+& \int_0^t\alpha(\rho^\alpha_s(0)-\rho^\alpha_s(1))(H_s(0)-H_s(1))\,ds\;=\;0\;,
\end{split}
\end{equation}
where $\partial_u\rho^\alpha$ is the weak derivative of  $\rho^\alpha$.
Our goal  consists in  analyzing  the limit, as  $\alpha\to 0$ or $\alpha\to\infty$, of the terms in the previous equation.
Due to  boundary restrictions, the last integral term above is analyzed separately. Moreover, Proposition \ref{lemmapdePT1} covers the case  $\alpha\to 0$ and Proposition \ref{lemmapdePT2} covers the case $\alpha\to \infty$.
We begin by showing some smoothness of a weak solution of \eqref{her} that will be needed in order to take limits.

\begin{proposition}\label{novolema}
For any  $H\in\A$, there exists a constant $C_H^T$ not depending on $\alpha$ such that
\begin{equation*}
\vert\,\< \rho^\alpha_t, \, H_t\,\>\,  -\,\< \rho^\alpha_s, \, H_s\,\>\,\vert \leq C_H^T |\,t-s\,|^{1/2}\,,\qquad\forall s,t\in [0,T]\,.
\end{equation*}
\end{proposition}

\begin{proof}
Let $H\in\A$. Since $\rho^\alpha$ satisfies the integral equation  \eqref{eqintbyparts}, it is sufficient to estimate the absolute value of:
\begin{equation*}
\begin{split}
&R_1:= \int_s^t\<\, \partial_u\rho^\alpha_r , \,\partial_u  H_r\>\,dr\,,\\
&R_2:=\int_s^t\< \rho^\alpha_r, \, \partial_r H_r\,\>\, dr\,,\\
&R_3:= \int_s^t\alpha(\rho^\alpha_s(0)-\rho^\alpha_s(1))(H_s(0)-H_s(1))\,dr\,.
\end{split}
\end{equation*}

We start by the case $\alpha\geq 1$. At first we notice that Proposition \ref{fund} guarantees that $R_3$ can be rewritten\footnote{This is essentially the Fundamental Theorem of Calculus by seeing the unit interval as the torus. The proof is technical and is postponed to the Appendix.} as
\begin{equation*}
 \begin{split}
& \int_s^t\partial_u\rho^\alpha_r(0)(H_r(0)-H_r(1))\,dr\;.
\end{split}
\end{equation*}
By the Cauchy-Schwarz inequality,
\begin{equation*}
|R_3|\leq{\Big(\int_0^T(\partial_u\rho^\alpha_r(0))^2\,dr\Big)^{1/2}2\Vert H\Vert_\infty |t-s|^{1/2}}\;.
\end{equation*}
Since $\alpha\geq 1$, then $\<\!\<H,H\>\!\>_\alpha \leq \<\!\<H,H\>\!\>_1$. As a consequence of Proposition \ref{uniq_sobolev}, the function $\rho^\alpha$ satisfies
\begin{equation*}
 \<\!\<\rho^{\alpha},\partial_u H\>\!\>-2\<\!\<H,H\>\!\>_{1}\leq K_0\,,
\end{equation*}
for all $H\in C^{\,0,1}([0,T]\times\bb T)$. Thus, by Proposition \ref{lema3Wa} we conclude that
\begin{equation*}
 \begin{split}
& \int_0^T(\partial_u\rho^\alpha_r(0))^2\,dr\leq 8K_0\,,
\end{split}
\end{equation*}
from where we get that
$$|R_3|\leq{(8K_0)^{1/2}\,2\Vert H\Vert_\infty |t-s|^{1/2}}.$$

Analogously, by the Cauchy-Schwarz inequality, Proposition \ref{uniq_sobolev} and Proposition \ref{lema3Wa},
$$|R_1|\leq{(8K_0)^{1/2}\,2\Vert \partial_u H\Vert_\infty |t-s|^{1/2}}\,.$$

Finally, $R_2$
can be easily bounded from above by $\Vert  \partial_r H\Vert_\infty |t-s|$.

The case $\alpha<1$ is easier. Since to estimate $R_1$ and $R_2$ we did not impose any restriction on $\alpha$,
it remains to estimate $R_3$, which is  bounded from above by $4\Vert H\Vert_\infty |t-s|$.

To complete the bounds, notice that $ |t-s|\leq{T^{1/2} |t-s|^{1/2}}$, which is true because $0\leq s,t\leq T$.

\end{proof}

Now, we analyze the limit of the terms in the integral equation \eqref{eqintbyparts} along a subsequence $\rho^{\alpha_n}$.
Provided by the next proposition, we will be able to replace $\rho_t^{\alpha_n}$ by its  $L^2$-limit in the first, second and fourth terms of the integral equation \eqref{eqintbyparts}.
In fact, we will need to take the limit along a subsequence of $\alpha_n$. However, since we aim the uniqueness of limit points, this is not a problem.
\begin{proposition}\label{lemmaintegrais0}
 Suppose that $\rho^{\alpha_n}$ converges to $\rho^*$ in $L^2([0,T]\times [0,1])$, as  $n\to\infty$.
Then, there exists a function $\tilde{\rho}$ such that $\rho^*=\tilde{\rho}$ almost surely and   $t\mapsto\<\,\tilde{\rho}_t,\,H_t\,\>$ is a continuous map.
Moreover, there exists a subsequence $n_j$ such that
\begin{equation*}
 \lim_{j\to\infty} \<\,\rho^{\alpha_{n_j}}_t,\,H_t\,\>\,=\,\<\,\tilde{\rho}_t,\,H_t\,\>\,,
\end{equation*}
for all $t\in[0,T]$ and for all $H\in\A$.
\end{proposition}
\begin{proof}


For $H\in\A$ and  $n\in \bb N$ consider the function
 \begin{equation*}
 \begin{split}
 f_n(\cdot,H):\,&[0,T]\rightarrow{\bb R}\\
&t\mapsto\<\rho^{\alpha_n}_t,\,H_t\>\,.
\end{split}
\end{equation*}
By Proposition \ref{novolema}, the sequence $\{f_n(\cdot,H)\}_{n\in\bb N}$ is uniformly H\"older, hence equicontinuous. Since
$\vert f_n(t,H)\vert\leq \Vert H\Vert_\infty$, by the Arzel\`a-Ascoli Theorem,
there exists a subsequence $n_k$, depending on $H$, such that $f_{n_k}(\cdot, H)$ converges uniformly in $t$, as $k\to\infty$, to a continuous function $f(\cdot, H)$.

Since $\A$ is separable,
 applying a diagonal argument we can find a subsequence $n_j$
   such that the convergence above  along $n_j$ holds uniformly in $t$, for any function on a countable dense set of $\A$.
By density, the operator $f(t,\cdot)$ can be extended to a bounded linear functional in $ C^2([0,1])$ with respect to  the $L^2$ norm, which in turn can be extended to a bounded linear functional  in $L^2[0,1]$.

The  Riesz Representation Theorem implies the existence of a function $\tilde{\rho}_t\in L^2[0,1]$
such that $f(t,H)=\<\tilde{\rho}_t,H_t\>$. Notice that last equality holds for all $t\in[0,T]$.
Uniqueness of the limit ensures that $\rho^*=\tilde{\rho}$ almost surely.
\end{proof}
We point out that the hypothesis about the convergence in $L^2$  in the proposition above has no special importance. A convergence in other norm would work as well.  The $L^2$ norm plays a role indeed on the relative compactness of $\{\rho^\alpha:\,\alpha>0\}$.

The next proposition allows us to replace $\rho_t^{\alpha_n}$ by its limit as $n\to\infty$ in the third term of equation \eqref{eqintbyparts}.
\begin{proposition}\label{lemmaintegrais2}
 Suppose that $\rho^{\alpha_n}$ converges to $\rho^*$ in $L^2([0,T]\times [0,1])$.
Then, for all $t\in[0,T]$ and for all $H\in \A$,
\begin{equation*}
\lim_{n\to\infty}\int_0^t\<\,\partial_u\rho^{\alpha_n}_s,\,\partial_uH_s\,\>\,ds\,=\, \int_0^t\<\,\partial_u\rho^{*}_s,\,\partial_uH_s\,\>\,ds\,.
\end{equation*}
\end{proposition}

\begin{proof}
If $H\in\A$, then $\partial_uH$ belongs to the set $C^{\,1,1}([0,T]\times[0,1])$. For this reason, the proof is written in terms of functions belonging to this last domain.

Fix a time $t$.
Consider first $H\in C^{\,0,1}([0,T]\times[0,1])$ compactly supported in $[0,t]\times(0,1)$.
  In this case,
\begin{equation*}
\int_0^t\<\,\partial_u\rho^{\alpha_n}_s,\,H_s\,\>\,ds=- \int_0^t\<\,\rho^{\alpha_n}_s,\,\partial_u H_s\,\>\,ds\,.
\end{equation*}
because the integrands above vanish for times greater than $t$.
Since $\rho^{\alpha_n}$ converges to $\rho^*$ in $L^2([0,T]\times [0,1])$, the previous equality shows that
\begin{equation*}
 \lim_{n\to\infty}\int_0^t\<\,\partial_u\rho^{\alpha_n}_s,\,H_s\,\>\,ds=\int_0^t\<\,\partial_u\rho^{*}_s,\, H_s\,\>\,ds\,.
\end{equation*}
The next step is to extend the previous equality to functions without that condition on the support.
 Let $H\in C^{1,1}([0,T]\times[0,1])$ and approximate this function  in $L^2([0,T]\times[0,1])$   by a
function $H^\eps\in C^{1,1}([0,T]\times[0,1])$ with compact support in $[0,T]\times(0,1)$ and such that
$\Vert H^\eps\Vert_\infty\leq\Vert H\Vert_\infty$.
For $\delta>0$, let us define the function $\varphi^\delta:[0,T]\to\bb R$ as
\begin{equation*}
 \varphi^\delta(s)= \left\{
\begin{array}{ll}
1,& \mbox{if}\,\,s\in[0,t-\delta]\,, \\
\displaystyle\frac{t-s}{\delta},& \mbox{if}\,\,s\in[t-\delta,t]\,, \\
0, & \mbox{if}\,\,s\in[t, T]\,.\\
\end{array}
\right.
\end{equation*}
Let $H^{\eps,\delta}_s(u):=H^\eps_s(u)\varphi^\delta(s)$. Then,
$ H^{\eps,\delta}\in C^{\,0,1}([0,T]\times[0,1])$ and has  compact support contained in $[0,t]\times(0,1)$.
Hence, from what we proved above,
\begin{equation}\label{imp}
\begin{split}
\lim_{n\to \infty} \int_0^t\<\, \partial_u\rho^{\alpha_n}_s, \,H^{\eps,\delta}_s\,\>\, ds
=
 \int_0^t\<\, \partial_u\rho^{*}_s, \, H^{\eps,\delta}\,\>\, ds\,.
\end{split}
\end{equation}
By the triangular inequality,
\begin{equation}\label{eqintbyparts*}
\begin{split}
\Big\vert\!\int_0^t\!\< \partial_u\rho^{\alpha_n}_s-\partial_u\rho^{*}_s, H_s\>ds&\,\Big\vert\!\leq{\Big\vert\!\int_0^t\!\< \partial_u\rho^{\alpha_n}_s, H_s-H^{\eps}_s\> ds\Big\vert}+\Big\vert\!\int_0^t\!\< \partial_u\rho^{\alpha_n}_s,H^{\eps}_s-H_s^{\eps,\delta}\> ds\Big\vert\\
&+\Big\vert\!\int_0^t\!\<\partial_u\rho^{\alpha_n}_s-\partial_u\rho^{*}_s, H^{\eps,\delta}_s\> ds\Big\vert+\Big\vert \!\int_0^t\!\<\partial_u\rho^{*}_s, H_s^{\eps,\delta}-H^{\eps}_s\> ds\Big\vert\\
&+
\Big\vert\!\int_0^t\<\partial_u\rho^{*}_s, H^{\eps}_s-H_s\> ds\Big\vert\,.
\end{split}
\end{equation}
We shall estimate each term on the right hand side of the previous inequality. We start by the first one.
By the Cauchy-Schwarz inequality,
\begin{equation*}
\begin{split}
&\Big\vert\int_0^t\!\!\<\, \partial_u\rho^{\alpha_n}_s, H_s-H^{\eps}_s\,\>\, ds\,\Big\vert
\leq\Big(\int_0^t\!\!\Vert \partial_u\rho^{\alpha_n}_s\Vert^2\, ds\Big)^{1/2}
\Big(\int_0^t\!\!\Vert H_s-H^{\eps}_s\Vert^2\, ds\Big)^{1/2}.\\
\end{split}
\end{equation*}
We notice that by Proposition \ref{uniq_sobolev}, $\rho^{\alpha_n}$ satisfies
\begin{equation*}
 \int_0^t\<\rho^{\alpha_n}_s,\,\partial_u G_s\,\>\,ds-2\int_0^t\<G_s,\,G_s\,\>\,ds\leq K_0\,,
\end{equation*}
 for all $H\in C_c$, see Definition \ref{def9}, and any $t\in{[0,T]}$. This  together with Corollary \ref{lema3} ensures that $\int_0^t\Vert \partial_u\rho^{\alpha_n}_s\Vert^2\, ds\leq 8K_0\,.$
 Thus,
\begin{equation*}
\begin{split}
&\Big\vert\int_0^t\<\, \partial_u\rho^{\alpha_n}_s, \,H_s-H^{\eps}_s\,\>\, ds\Big\vert
\leq(8K_0)^{1/2}
\Big(\int_0^t\Vert H_s-H^{\eps}_s\Vert^2\, ds\Big)^{1/2} .\\
\end{split}
\end{equation*}
By Proposition \ref{propositL2}, the same holds for $\rho^*$, i.e.,
$\int_0^t\Vert \partial_u\rho^{*}_s\Vert^2\, ds\leq 8K_0$, hence the previous inequality follows replacing $\rho^{\alpha_n}$ by $\rho^*$. With this we also estimated the last term on the right hand side of the previous inequality.
Now we estimate the second term on the right hand side of \eqref{eqintbyparts*}. Observe that
\begin{equation*}
\begin{split}
 \Big\vert\int_0^t\<\, \partial_u\rho^{\alpha_n}_s, \,H_s^{\eps,\delta}-H^{\eps}_s\,\>\, ds\Big\vert&=
 \Big\vert\int_0^t\int_{0}^1  \partial_u\rho^{\alpha_n}_s(u)[H^\eps_s(u)\varphi^\delta(s)-H^{\eps}_s(u)]\,du\, ds\Big\vert\\
&=\Big\vert\int_{t-\delta}^t(\pfrac{t-s}{\delta}-1)\int_{0}^1  \partial_u\rho^{\alpha_n}_s(u)\,H^\eps_s(u)\,du\,ds\Big\vert\\
&\leq  {\int_{0}^T {\bf{1}}_{[t-\delta,t]}(s)
\Big\vert\int_{0}^1  \partial_u\rho^{\alpha_n}_s(u)\,H^\eps_s(u)\,du\Big\vert\,\,ds}\,.
\end{split}
\end{equation*}
By the Cauchy-Schwarz inequality we obtain that
\begin{equation*}
\begin{split}
\Big\vert\int_0^t\<\, \partial_u\rho^{\alpha_n}_s, \,H_s^{\eps,\delta}-H^{\eps}_s\,\>\, ds\Big\vert\leq&  \sqrt{\delta}\Big(\int_{0}^T\Big\vert
\int_{0}^1  \partial_u\rho^{\alpha_n}_s(u)\,H^\eps_s(u)\,du\Big\vert^2\,\,ds\Big)^{1/2}\\
\leq &\sqrt{\delta}\Vert H^\eps\Vert_{\infty}\Big(\int_{0}^T\Big\vert
\int_{0}^1  \partial_u\rho^{\alpha_n}_s(u)\,du\Big\vert^2\,\,ds\Big)^{1/2}\\
\leq& \sqrt{\delta}\Vert H\Vert_{\infty}(8K_0)^{1/2}\,.\\
\end{split}
\end{equation*}
By analogous calculations, we also get the previous estimate replacing $\rho^{\alpha_n}$ by $\rho^*$. Therefore we also estimated the fourth term on the right hand side of \eqref{eqintbyparts*}.

Putting together the previous computations, we obtain that the left hand side of \eqref{eqintbyparts*} is bounded from above by
\begin{equation*}
\Big\vert\int_0^t\<\, \partial_u\rho^{\alpha_n}_s- \partial_u\rho^{*}_s, \, H^{\eps,\delta}\,\>\, ds\,\Big\vert
 +2(8K_0)^{1/2}\Big\{\Big(\int_0^t\Vert H_s-H^{\eps}_s\Vert^2 ds\Big)^{1/2}  \!\!\!+ \!\sqrt{\delta}\Vert H\Vert_{\infty} \Big\}\,.
\end{equation*}
Employing \eqref{imp}, recalling  the definition of $H_\varepsilon$, sending  $n\to\infty$, and then $\varepsilon,\delta$ to zero,   the proof ends.
\end{proof}

Finally, in the next two propositions we are able to identify the integral equations for the limit of $\rho^{\alpha_n}$ when $\alpha_n\to 0$ or $\alpha_n\to \infty$ by treating the last term of the integral equation \eqref{eqintbyparts}. We start by showing that the limit of $\rho^{\alpha_n}$ when $\alpha_n\to 0$ is a weak solution of the heat equation with Neumann's  boundary conditions.
\begin{proposition}\label{lemmapdePT1}
Let $\{\alpha_n\}_{n\in \bb N}$ be a sequence of positive real numbers such that $\lim_{n\to\infty} \alpha_n\;=\;0\,.$
If $\{\rho^{\alpha_n}\}_{n\in\bb N}$ converges to $\rho^*$ in $L^2([0,T]\times [0,1])$,  then
$\rho^*$ is the unique weak solution of \eqref{hen}.
\end{proposition}
\begin{proof}
 Proposition \ref{propositL2} says that  $\rho^*\in L^2(0,T;\mc H^1)$, which is one of the conditions in Definition \ref{heat equation Neumann}.

 In order to prove that $\rho^*$ satisfies \eqref{eqint3}, the idea is to take the limit as  $n\to\infty$ in  \eqref{eqintbyparts} and to analyze the limiting terms.
 By the previous propositions, it only remains to analyze the limit of the last term in the integral equation \eqref{eqintbyparts}.

A simple computation shows that for $t\in{[0,T]}$:
\begin{equation*}
\Big\vert\int_0^t \alpha_n (\rho^{\alpha_n}_s(0)-\rho^{\alpha_n}_s(1))(H_s(0)-H_s(1))\,ds\Big\vert\;\leq
\; 4 T \Vert H\Vert_\infty \, \alpha_n\,,
\end{equation*}
therefore, when $\alpha_n\to 0$, last integral in \eqref{eqintbyparts}  converges to zero, as $n\to\infty$.
Therefore, replacing $\rho^\alpha$ by $\rho^{\alpha_n}$ in \eqref{eqintbyparts} and taking the limit,  we  conclude that $\rho^*$ satisfies:
\begin{equation*}
\< \,\rho^*_t,\,H_t\,\> -\<\,\rho_0,\,H_0\,\>+ \int_0^t\<\, \partial_u\rho^*_s, \, \partial_u  H_s\,\>\, ds-
\int_0^t\<\, \rho^*_s, \, \partial_s H_s\,\>\, ds\;=\;0\;,
\end{equation*}
for all $t\in [0,T]$  and  for all $H\in \A$.

Since $\rho^*\in L^2(0,T;\mc H^1)$, performing an integration by parts in the previous equation, we get to:
\begin{equation*}
\<\rho^*_t,H_t\>\! -\!\<\rho_0,H_0\>- \!\int_0^t\!\< \rho^*_s,\Delta  H_s+\partial_s H_s\>ds\\
-\int_0^t\!\!\!\!(\rho^*_s(0)\partial_uH_s(0)-\rho^*_s(1)\partial_uH_s(1))ds\!\!=\!\!0.
\end{equation*}
for all $t\in [0,T]$  and  for all $H\in\A$, concluding the proof.
\end{proof}

In next proposition we treat the last term of the integral equation \eqref{eqintbyparts}.
\begin{proposition}\label{lemmapdePT2}
Let $\{\alpha_n\}_{n\in \bb N}$ be a sequence of positive real numbers such that $\lim_{n\to\infty} \alpha_n\;=\;\infty\,.$
If  $\{\rho^{\alpha_n}\}_{n\in\bb N}$ converges to $\rho^*$ in $L^2([0,T]\times [0,1])$, then
$\rho^*$ is the unique weak solution of \eqref{he}.
\end{proposition}
\begin{proof}
 Proposition \ref{propositL2} says that  $\rho^*\in L^2(0,T;\mc H^1)$, which is one of the conditions in Definition \ref{def edp 1}.
We shall prove that  $\rho^*$ satisfies \eqref{eqint1}. As before, the idea is to take the limit as  $n\to\infty$ in  \eqref{eqintbyparts} and to analyze the limiting terms. In this situation, we take  $H\in C^{1,2}([0,T]\times\bb T)$, so that
\eqref{eqintbyparts} is given by
\begin{equation*}
\begin{split}
& \<\rho^{\alpha_n}_t,\,H_t\,\> -\<\rho_0,\,H_0\,\>+\int_0^t\<\, \partial_u\rho^{\alpha_n}_s , \,\partial_u  H_s\,\>\, ds-
\int_0^t\< \rho^{\alpha_n}_s, \, \partial_s H_s\,\>\, ds\;=\;0\,.
\end{split}
\end{equation*}
By the first statement of Proposition \ref{lemmaintegrais0} and Proposition \ref{lemmaintegrais2}, taking the limit as $n\to \infty$ in the previous equality, we  conclude that $\rho^*$ satisfies:
\begin{equation*}
\begin{split}
& \<\rho^*_t,\,H_t\,\> -\<\rho_0,\,H_0\,\>+\int_0^t\<\, \partial_u\rho^*_s , \,\partial_u  H_s\,\>\, ds-
\int_0^t\< \rho^*_s, \, \partial_s H_s\,\>\, ds\;=\;0\;,
\end{split}
\end{equation*}
for all $t\in[0,T]$ and  for all $H\in C^{1,2}([0,T]\times\bb T)$.
To obtain the integral equation \eqref{eqint1} from the  equation above, we invoke Proposition \ref{lemmaintegrais2} and perform an integration by parts, we are lead to
\begin{equation*}
\begin{split}
\int_0^t\<\, \partial_u\rho^*_s , \,\partial_u  H_s\,\>\, ds&=
\!\!\lim_{\alpha_n\to\infty}\int_0^t\<\, \partial_u\rho^{\alpha_n}_s , \partial_u  H_s\,\>\, ds\\
&=\!\!\lim_{\alpha_n\to\infty}\!\int_0^t\<\rho^{\alpha_n}_s ,\Delta H_s\> ds -\!\!
\int_0^t\!\!\!\Big(\rho^{\alpha_n}_s(0)-\rho^{\alpha_n}_s(1)\Big)\partial_uH_s(b)\,ds.
\end{split}
\end{equation*}
We claim that the previous  limit is equal to
$\int_0^t\<\rho^{*}_s , \,\Delta H_s\,\>\, ds$. At first we prove  that
\begin{equation*}
\begin{split}
\lim_{\alpha_n\to\infty}
\int_0^t\Big(\rho^{\alpha_n}_s(0)-\rho^{\alpha_n}_s(1)\Big)\partial_uH_s(0)\,ds\,=\,0\,.
\end{split}
\end{equation*}
By the Cauchy-Schwarz inequality and Proposition \ref{fund},
\begin{equation*}
\begin{split}
\int_0^t\!\!\Big(\rho^\alpha_s(0)-\rho^\alpha_s(1)\Big)\partial_uH_s(0)\,ds\leq
\Big(\int_0^T\!\!(\partial_uH_s(0))^2\,ds\Big)^{1/2}\frac{1}{\alpha} \Big(\int_0^T\!\!(\partial_u\rho^\alpha_s(0))^2\,ds\Big)^{1/2}\!\!,
\end{split}
\end{equation*}
for all $t\in{[0,T]}$.
Without loss of generality, we can assume $\alpha\geq 1$. Thus, by  Proposition \ref{uniq_sobolev} and the fact that $\<\!\<H,H\>\!\>_{\alpha}\leq{\<\!\<H,H\>\!\>_{1}}$ because $\alpha\geq{1}$, we arrive at
\begin{equation*}
 \<\!\<\rho^{\alpha},\partial_u H\>\!\>-2\<\!\<H,H\>\!\>_{1}\leq K_0\,,
\end{equation*}
for all $H\in C^{\,0,1}([0,T]\times\bb T)$. From Proposition \ref{lema3Wa},
\begin{equation*}
 \sup_{H}\Big\{\<\!\<\rho^{\alpha},\partial_u H\>\!\>-2\<\!\<H,H\>\!\>_{1}\Big\}=
\frac{1}{8}\int_0^T\Big\{
\Vert \partial_u\rho^\alpha_s\Vert^2+(\partial_u\rho^\alpha_s(0))^2\Big\}\,ds\,.
\end{equation*}
where the supremum above is taken over functions $H\in C^{\,0,1}([0,T]\times\bb T)$, see Definition \ref{space C^n,m}.
Therefore,
\begin{equation*}
\begin{split}
\int_0^t(\rho^\alpha_s(0)-\rho^\alpha_s(1))\partial_uH_s(0)\,ds\leq \frac{1}{\alpha} (8K_0)^{1/2}
\Big(\int_0^T(\partial_uH_s(0))^2\,ds\Big)^{1/2}\,,
\end{split}
\end{equation*}
for all $t\in{[0,T]}$ and $\alpha\geq 1$.

In order to finish the proof is is enough to show that
\begin{equation*}
\lim_{\alpha_n\to\infty}\int_0^t\<\rho^{\alpha_n}_s , \,\Delta H_s\,\>\, ds=\int_0^t\<\rho^{*}_s , \,\Delta H_s\,\>\, ds,
\end{equation*}
which is  consequence of the Cauchy-Schwarz inequality.

\end{proof}

\begin{proof}[Proof of Theorem \ref{pdePT}]
As mentioned after Proposition \ref{propositL1}, the set $\{\rho^\alpha:\,\alpha>0\}$
is relatively compact in $L^2([0,T]\times [0,1])$.
Therefore, any sequence $\alpha_n\to 0$ has a subsequence $\alpha_{n_j}$ such that $\rho^{\alpha_{n_j}}$ converges to some
$\rho^*$. By Proposition \ref{propositL2}, Proposition \ref{lemmapdePT1} and from uniqueness of weak solutions of  \eqref{hen},
we conclude that $\rho^*$ is the unique weak solution of \eqref{hen}. Hence, $\lim_{\alpha\to 0}\rho^\alpha=\rho^*$. Analogously, employing Proposition \ref{lemmapdePT2},  we get that $\lim_{\alpha\to \infty}\rho^\alpha=\hat{\rho}\,,$ where $\hat{\rho}$ is the unique weak solution of  \eqref{he}.
\end{proof}


\appendix
\section{Sobolev Space Tools}

We prove here some results that we  used along the paper. Most of them are suitable applications of the Riesz Representation Theorem for Hilbert spaces.

\begin{proposition}
 The set $C^{\,0,1}([0,T]\times \bb T)$ is a dense subset of  $L^2_{W_\alpha}([0,T]\times \bb T)$.
\end{proposition}
\begin{proof}
 Let $H\in L^2_{W_\alpha}([0,T]\times \bb T)$. Then, $H\in L^2([0,T]\times \bb T)$ and $H(\cdot,0)\in L^2([0,T])$. Consider a sequence $\{H_n\}_{n\in \bb N}$ such that for each $n\in{\bb N}$, $H_n\in C^{\,0,1}([0,T]\times \bb T)$  with compact support in $[0,T]\times (\bb T\backslash \{0\})$
converging in $L^2([0,T]\times \bb T)$ to $H$, as $n\to\infty$. Consider also a sequence $\{h_n\}_{n\in \bb N}$, of continuous functions $h_n:[0,T]\rightarrow{\bb R}$ and
 converging in  $L^2([0,T])$ to $H(\cdot, 0)$, as $n\to\infty$. For each $n\in\bb N$, let
$$G_n(t,u):=H_n(t,u)+h_n(t)\textbf{1}_{(-\frac{1}{n},\,\frac{1}{n})}(u)\,.$$
Noticing that, for each $n\in\bb N$,
\begin{equation*}
 \Vert G_n-H\Vert_{{\alpha}}^2= \Vert H_n-H\Vert^2+\frac{2}{\alpha\,n}\int_0^T(H(t,0)-h_n(t))^2\, dt
\end{equation*}
the proof ends.
\end{proof}

\begin{proposition}\label{lemader_W}
Let $\xi:[0,T]\times\bb T\to \bb R$ be such that
$$\sup_{H}\Big\{\<\!\< \partial_u H,\xi \>\!\>- \kappa \<\!\< H,H\>\!\>_{\alpha}\Big\}<\infty,$$
for some $\kappa>0$, where the supremum is taken over functions $H\in C^{\,0,1}([0,T]\times\bb T)$.
Then,
there exists a function  in $L^2_{W_\alpha}([0,T]\times \bb T)$, which we denote by $\partial_u\xi$, such that
\begin{equation}\label{der_W}
 \<\!\< \partial_u H,\xi \>\!\>=- \<\!\< H,\partial_u\xi\>\!\>_{\alpha}
=-\<\!\< H,\partial_u\xi\>\!\>-\frac{1}{\alpha}\int_0^T H_t(0)\partial_u\xi_t(0)\,dt\,,
\end{equation}
for all $H\in C^{\,0,1}([0,T]\times\bb T)$.
\end{proposition}
\begin{proof}
Following the same arguments as in the proof of Proposition \ref{propositL2}, this is a consequence of Riesz Representation Theorem.
\end{proof}
\begin{remark}\label{remark1}
 The function $\partial_u\xi$ above
  is indeed the weak derivative of the function $\xi$ in the usual sense. To see this, notice that, for $H\in C_c$,
\begin{equation*}
  \<\!\< \partial_u H,\xi \>\!\>=- \<\!\< H,\partial_u\xi\>\!\>_{\alpha}=-\<\!\< H,\partial_u\xi\>\!\>\,.
\end{equation*}
\end{remark}
\begin{proposition}\label{lema3Wa}
Let $\xi:[0,T]\times\bb T\to \bb R$ be such that
$$\sup_{H}\,\Big\{\<\!\< \partial_u H,\xi \>\!\>- \kappa \<\!\< H,H\>\!\>_{\alpha}\,\Big\}<\infty\,,$$
for some $\kappa>0$, where the supremum is taken over functions $H\in C^{\,0,1}([0,T]\times\bb T)$.  Then,
\begin{equation}\label{energy exp}
\begin{split}
 \sup_{H}\,\Big\{ \<\!\< \partial_u H,\xi \>\!\>- \kappa \<\!\< H,H\>\!\>_{\alpha}\Big\}\,&=
\,\pfrac{1}{4\kappa}\int_0^T\Vert\partial_u\xi_t\Vert_{{\alpha}}^2\,dt\\
&=
\,\pfrac{1}{4\kappa}\int_0^T\Big(\Vert\partial_u\xi_t\Vert^2+\frac{1}{\alpha}(\partial_u\xi_t(0))^2\Big)\,dt\,,
\end{split}
\end{equation}
where the supremum is taken over functions $H\in C^{\,0,1}([0,T]\times\bb T)$.
\end{proposition}

\begin{proof}
By Proposition \ref{lemader_W}, for all $H\in C^{\,0,1}([0,T]\times\bb T)$,
\begin{equation}\label{eq40}
 \<\!\< \partial_u H,\xi \>\!\>- \kappa \<\!\< H,H\>\!\>_{\alpha}\,=\,-\<\!\<  H,\partial_u\xi \>\!\>_{\alpha}- \kappa \<\!\< H,H\>\!\>_{\alpha}\,.
\end{equation}
By Young's Inequality, for all $r>0$,
\begin{equation*}
 \vert\<\!\<  H,\partial_u\xi \>\!\>_{\alpha}\vert\;\leq\; \frac{r}{2}\<\!\<  H,H \>\!\>_{\alpha}
+\frac{1}{2r}\<\!\<  \partial_u\xi,\partial_u\xi \>\!\>_{\alpha}\,.
\end{equation*}
Choosing $r=2\kappa$, together with \eqref{eq40} we get to
 \begin{equation*}
\sup_{H}\,\Big\{ \<\!\< \partial_u H,\xi \>\!\>- \kappa \<\!\< H,H\>\!\>_{\alpha}\Big\}
\leq\pfrac{1}{4\kappa}\int_0^T\Vert\partial_u\xi_t\Vert_{{\alpha}}^2\,dt\,.
 \end{equation*}
For the reversed inequality, let
$\{H^n\}_{n\in\bb N}\subset C^{\,0,1}([0,T]\times\bb T)$
 converging to $r\,\partial_u\xi$
in $L^2_{W_\alpha}([0,T]\times \bb T)$, as $n\to \infty$. The constant $r\in \bb R$ will be chosen ahead. Thus,
 \begin{equation*}
\begin{split}
\sup_{H}\;\Big\{\<\!\<  \partial_u H,\xi \>\!\>- \kappa \<\!\< H,H\>\!\>_{\alpha}\Big\}=&
\sup_{H}\;\Big\{-\<\!\<  H,\partial_u\xi \>\!\>_{\alpha}- \kappa \<\!\< H,H\>\!\>_{\alpha}\Big\}\\
\geq&
\lim_{n\to\infty} -\<\!\< H^n,\partial_u\xi \>\!\>_{\alpha}- \kappa \<\!\<H^n,H^n\>\!\>_{\alpha}\\
=&(-r-\kappa r^2)\int_0^T\Vert\partial_u\xi_t\Vert_{{\alpha}}^2\,dt\,.
\end{split}
 \end{equation*}
Taking $r=-\frac{1}{2\kappa}$ the proof ends.

\end{proof}

\begin{corollary}\label{lema3}
Let $\xi\in L^2(0,T;\mc H^1)$ be such that
 $$\sup_{H}\Big\{\<\!\< \partial_u H,\xi \>\!\>- \kappa \<\!\< H,H\>\!\>\Big\}<\infty\,,$$
 for some $\kappa>0$. Then,
\begin{equation*}
\sup_{H}\,\Big\{ \<\!\< \partial_u H,\xi \>\!\>- \kappa \<\!\< H,H\>\!\>\Big\}\,=
\,\pfrac{1}{4\kappa}\int_0^T\Vert\partial_u\xi_t\Vert^2\,dt\,,
\end{equation*}
where  the function $\partial_u\xi_t$ coincides, Lebesgue almost surely, with the  function $\partial_u\xi_t$ of Proposition \ref{lema3Wa}.
Above, the supremums are taken over $H\in C_c$, see Definition \ref{def9}.
\end{corollary}

\begin{proof}
 Since $\xi\in L^2(0,T;\mc H^1)$ and by Remark \ref{remark1}, the result is a consequence of  Proposition \ref{lema3Wa}.
\end{proof}

\begin{proposition}\label{fund}
Let $\xi: [0,T]\times [0,1]\to \bb R$ be such that $$\sup_{H}\<\!\< \partial_u H,\xi \>\!\>- \kappa \<\!\< H,H\>\!\>_{\alpha}<\infty,$$ for some $\kappa>0$,
where the supremum is taken over functions $H\in C^{\,0,1}([0,T]\times\bb T)$. Then, $t\in[0,T]$ almost surely,
\begin{equation*}
\xi_t(v)-\xi_t(u)=\int_{[u,v)} \partial_u \xi_t(z)\,W_\alpha(dz),\quad \forall u,v\in \bb T,
\end{equation*}
where $\partial_u\xi$ satisfies \eqref{der_W}. In particular,  $t\in[0,T]$ almost surely
\begin{equation*}
\xi_t(0)-\xi_t(1)=\frac{1}{\alpha} \partial_u \xi_t(0)\,.
\end{equation*}

\end{proposition}
\begin{proof} A function defined on the interval $[0,1]$ can be identified, Lebesgue almost surely, with a function defined in the continuous torus $\bb T$.
It is in this sense that the function $\xi_t$ is understood here. Besides that, it is fixed the following orientation on $\bb T$. If $u, v\in (0,1]$ and $u<v$, then the integral over $(v,u]$ corresponds to a integral over $(v,1]\cup(0,u]$.

From Proposition \ref{lemader_W}, for $t\in{[0,T]}$ almost surely,
\begin{equation}\label{sob}
\<\!\< \partial_uH,\,\xi_t\>\!\>=
 -\<\!\<H,\,\partial_u\xi_t\>\!\>_{\alpha }\,,
\end{equation}
for all $H\in C^{1}(\bb T)$.
For $u,v\in{\bb T}$ and $n\geq{1}$ define $f_n:\bb T\to \bb R$ by
\begin{equation*}
f_n(w)\;=\;\left\{\begin{array}{cc}
-W_\alpha([u,u+\frac{1}{n}])^{-1}, &  \mbox{if}\,\,\,\,w\in[u,u+\frac{1}{n}]\\
W_\alpha([v,v+\frac{1}{n}])^{-1}, &  \mbox{if}\,\,\,\,w\in[v,v+\frac{1}{n}]\\
0, &\mbox{otherwise,}
\end{array}
\right.
\end{equation*}
\noindent and $H_n:\bb T\to\bb R$ by
\begin{equation*}
H_n(z)=\int_{(0,z]} f_n(w)\, W_\alpha(dw)\,.
\end{equation*}
Notice that since  $f_n$ is not a continuous function, then $H_n\notin C^{1}(\bb T)$. However, approximating $f_n$ and $H_n$ by  continuous functions
$f_n^\eps: \bb T\to \bb R$ and $H_n^\eps(r)=\int_{(0,r]} f_n^\eps(z)\, W_\alpha(dz)$,
respectively, equality $\eqref{sob}$
is still valid for $f_n$ and $H_n$.

We claim that $H_n(r)$ converges to $-\textbf 1_{[u,v)}(r)$,
$W_\alpha$-almost surely, as $n\to\infty$.

Indeed, if $u,v\neq 0$,  $H_n(r)$ converges pointwisely, as $n\to\infty$, to $-\textbf 1_{(u,v]}(r)$,
which is equal to $-\textbf 1_{[u,v)}(r)$, $W_\alpha$-almost surely.

If $u=0$, then $H_n(r)$ converges
pointwisely, as $n\to\infty$, to $-\textbf 1_{[u,v]}(r)$, which is equal to
 $-\textbf 1_{[u,v)}(r)$, $W_\alpha$-almost surely.

 If $v=0$, then $H_n(r)$ converges
pointwisely, as $n\to\infty$, to $-\textbf 1_{(u,v)}(r)$, which is equal to $-\textbf 1_{[u,v)}(r)$, $W_\alpha$-almost surely.
By the Cauchy-Schwarz inequality,
$$\lim_{n\to\infty}\<-\partial_u\xi_t , H_n\>_{\alpha}=
\<\partial_u\xi_t , \textbf 1_{[u,v)}\>_{\alpha}\,.$$
By the definition of $f_n$, we have that for each $n\in\bb N$,
\begin{equation*}
 \begin{split}
\<\xi_t,f_n\>_{\alpha}=&\;\;\frac{1}{W_\alpha([v,v+\frac{1}{n}])}\int_{[v,v+1/n]}\xi_t(w)\,W_\alpha(dw)\\
&-\frac{1}{W_\alpha([u,u+\frac{1}{n}])}\int_{[u,u+1/n]}\xi_t(w)\,W_\alpha(dw)\,.
 \end{split}
\end{equation*}
Sending $n\to\infty$ in the previous equality, by Lebesgue-Besicovitch Differentiation Theorem
(see \cite{eg}) we obtain that $\<\xi_t , f_n\>_{\alpha}$ converges $W_\alpha$ -
 almost surely in $u,v\in\bb T$, to $\xi_t(v)-\xi_t(u)$, which finishes the proof of the first claim. Finally,
 \begin{equation*}
\xi_t(0)-\xi_t(1)=\lim_{\at{u\to 1^-}{v\to 0^+}} \int_{[u,v)} \partial_u \xi_t(w)\,W_\alpha(dw)=\frac{1}{\alpha} \partial_u \xi_t(0)\,.
\end{equation*}
\end{proof}

\section{Discussion on the heat equation with Robin's boundary conditions}

In this section we make some connections of the results we obtained here with respect to a particular case considered in \cite{fl}. From that paper, it is known that
the hydrodynamic equation  for the slowed symmetric exclusion presented in the case $\beta=1$, is given by
\begin{equation}\label{he2}
\begin{cases}
 \partial_t \rho(t,u) \; =\; \frac{d}{du}\frac{d}{dW} \rho(t,u)\,,\\
\rho(0,u) \;=\; \rho_0(u)\,,
\end{cases}
\end{equation}
where $\frac{d}{du}\frac{d}{dW}$ is a generalized derivative, being $W$ a measure given by the sum of the Lebesgue measure and a delta of Dirac. For the definition of the operator
 $\pfrac{d}{du}\pfrac{d}{dW}$, we refer to \cite{fl} and references therein. In this paper we found a classical description of last equation, namely the heat equation with Robin's boundary conditions as given in  \eqref{her}. Below we make some connections relating the solutions of these equations.

Firstly, we describe how to get the weak solution of \eqref{he2} from  the weak solution of equation \eqref{her}. Adapting from  \cite{fl} the definition of the set of test functions
for \eqref{he2}, we have the following definition:
\begin{definition}
Let $\mc H^1_{W_\alpha}$ be the set of functions $H$ in $L^2(\bb T)$ such that for $u\in{\bb T}$
\begin{equation}\label{H}
H(u) \;=\; \tilde a \;+\; \int_{(0,u]}\Big(\tilde b+\int_{0}^v h(w) \, dw\Big) W_\alpha(dv)\, ,
\end{equation}
for some function $ h$ in $L^2(\bb T)$ and $\tilde a,\tilde b\in{\mathbb{R}}$ such that
\begin{equation}\label{domain}
\int_{0}^{1}  h(u) \, du \;=\; 0\;, \quad
\int_{(0,1]}  \Big( \tilde b + \int_{0}^v h(w) \, dw \Big) W_\alpha(dv)\;=\;0\;,
\end{equation}
where  $W_\alpha$ was given in \eqref{W}.
\end{definition}
{
\begin{definition}
Let $\mc{C}_{W_{\alpha}}$ be the set of functions $H\in{\mc H^1_{W_\alpha}}$ such that $h\in{C(\bb T)}$.
\end{definition}}

We have the following property about the elements of the space $\mc{C}_{W_{\alpha}}$.

\begin{lemma}\label{lemma imp}
$\mc{C}_{W_{\alpha}}\subseteq{\A}$.
\end{lemma}

\begin{proof} From now on, we identify the torus $\bb T$ with $(0,1]$. Let $H\in\mc{C}_{W_{\alpha}}$.
In order to  prove that $H\in\A$, we restricted the domain $\bb T$  of $H$ to  the open interval $(0,1)$, according to the definition of the space $\A$.
By definition of the set $\mc{C}_{W_{\alpha}}$, for all $u\in(0,1)$, $H(u)$ can be written as
\begin{equation*}
\tilde a +\tilde b u+\int_{0}^u\int_{0}^v h(w)\,dw\,dv \,,
\end{equation*}
for some  function $ h\in C(\bb T)$ and some constants $\tilde a$ and $\tilde b$ in ${\mathbb{R}}$ satisfying  conditions \eqref{domain}.
Then, one can see that the restriction of $H$ to $(0,1)$ belongs to $C^{1,2}([0,T]\times (0,1))$. After, we need to construct an extension $\tilde{H}:[0,T]\times [0,1]\to \bb R$ such that:
\begin{itemize}
 \item $\tilde{H}\in C^{1,2}([0,T]\times[0,1])$;
 \item $\tilde H$ restricted to $[0,T]\times(0,1)$ coincides with $H$.
\end{itemize}
But, it is not hard to see that it is enough to consider the function
\begin{equation*}
\tilde{H} (u)\;=\; \tilde a +\tilde b u+\int_{0}^u\int_{0}^v h(w)\,dw\,dv
 \,,
\end{equation*}
defined for all $u\in[0,1]$.
\end{proof}

Moreover, all elements in $\mc{C}_{W_{\alpha}}$ satisfy the Robin's boundary conditions:

 \begin{lemma}
If $H\in \mc{C}_{W_{\alpha}}$, then $\partial_uH(0)=\partial_uH(1)=\alpha (H(0)-H(1))$.
 \end{lemma}
\begin{proof}
 Since $H\in \mc{C}_{W_{\alpha}}$, a simple computation shows that for $u\in{\bb T}\equiv(0,1]$
\begin{equation*}
H(u) \;=\;G(u) +\frac{\tilde b}{\alpha}\textbf{1}_{\{1\}}(u)\,,
\end{equation*}
 where
 \begin{equation*}
 G(u)=\tilde a+ \tilde{b} u+\int_{0}^u\int_{0}^v h(w) \, dw \, dv.
 \end{equation*}

 Notice that $G(\cdot)$ is continuous and smooth. Then, $H(0)=G(0)$ and $H(1)=G(1)+\frac{\tilde b}{\alpha}$. On the other hand,
$\partial_uH(0)=\partial_uH(1)=G'(0).$
Since $G'(0)=\tilde b$, then $\partial_uH(0)=\partial_uH(1)=\alpha(H(0)-H(1))$, which finishes the proof.
 \end{proof}

In \cite{fgn}, we considered $\alpha=1$ and we
proved that $\rho_t(\cdot)$ is a weak solution of \eqref{he2}, which in particular means that
\begin{equation}\label{eq-old}
\< \rho_t, H\> \;-\; \< \rho_0 , H\>
 -\int_0^t \big\< \rho_s , \pfrac{d}{du}\pfrac{d}{dW}  H \big\>\, ds\;=\;0\;,
\end{equation}
for all $t\in [0,T]$ and all $H\in \mc H^1_{W_1}$. Now, we present a result that relates the integral equations \eqref{eq-old} and \eqref{eqint2}.
We notice that by Proposition 6.3 of \cite{fgn}, it is enough to verify equation \eqref{eq-old} for functions in $\mc{C}_{W_\alpha}$.

\begin{proposition}
For $H\in \mc{C}_{W_\alpha}$, the integral equation \eqref{eqint2} coincides with the integral equation \eqref{eq-old}.
\end{proposition}
\begin{proof}
From Lemma \ref{lemma imp} we know that $\mc{C}_{W_{\alpha}}$
is a subset of $\A$, which is the space of test functions for  the integral equation \eqref{eqint2}.
 From the previous lemma,  for $H\in\mc{C}_{W_{\alpha}}$
the integral equation \eqref{eqint2} reads as
\begin{equation*}
\< \rho_t, H\> \;-\; \< \rho_0 , H\>
 -\int_0^t \big\< \rho_s , h \big\>\, ds\;=\;0\;,
\end{equation*}
where $h=\Delta H$.
Notice that a function in $\mc{C}_{W_{\alpha}}$ does not depend on time.
Now, it is enough to notice that from the definition of $\frac{d}{du}\frac{d}{dW} H=h$, see \cite{fgn}, we get that
\begin{equation*}
\big\< \rho_s , \Delta H \big\>\;=\;\big\< \rho_s , \pfrac{d}{du}\pfrac{d}{dW}  H \big\>\;.
\end{equation*}
\end{proof}
%
\begin{proposition}
 There exists a unique weak solution of \eqref{he} and of \eqref{hen}.
\end{proposition}
\begin{proof}
We start by showing uniqueness of \eqref{hen}. For that purpose, let $\rho_t$ be a weak solution of \eqref{hen} with $\rho_0(\cdot)\equiv{0}$.
 For $u\in{\bb T}$, let
$ H_k(u)\;=\;\sqrt{2}\cos(k\pi u)\,,\,k\in\bb N$.
Recalling the integral equation \eqref{eqint3}, for all $k\in\bb N$, $H_k\in C^2([0,1])$, $\p_u H_k(0)=\p_u H_k(1)=0$ and
\begin{equation*}
 \<\rho_t,H_k\>=-(k\pi)^2\int_0^t\<\rho_s, H_k\>ds.
\end{equation*}
 Now, by Gronwall's inequality it follows that  $\<\rho_t,H_k\>=0$, $\forall t>0$ and for all $k\in{\bb N}$.
Since $\{H_k\}_{k\in\bb N}$ is a complete orthonormal system in $L^2(\bb T)$, we obtain that $\rho_t(\cdot)\equiv 0$, $\forall t>0$.

 The uniqueness of weak solutions of \eqref{he} follows as above, but  considering instead the complete orthonormal system
$\{1\,,\sqrt{2}\cos(2k\pi u)\,,\,\sqrt{2}\sin(2k\pi u)\}_{k\in\bb N}$ composed of functions in $C^2(\bb T)$.
\end{proof}

\section*{Acknowledgements}
The authors would like to thank Milton Jara for important conversations on the subject.
PG thanks FCT (Portugal) for support through the research project ``Non-Equilibrium Statistical Physics" PTDC/MAT/109844/2009. PG thanks the Research Centre of Mathematics of the University of Minho, for the financial support provided by ``FEDER" through the ``Programa Operacional Factores de Competitividade  COMPETE" and by FCT through the research project PEst-C/MAT/UI0013/2011.
The authors thank FCT and Capes for the financial support through the research project ``Non-Equilibrium Statistical Mechanics of Stochastic Lattice Systems". The authors thank the
warm hospitality of IMPA, where this work was finished.

\bibliographystyle{amsplain}

\begin{thebibliography}{}

\bibitem{e} L. Evans.  \textit{Partial Differential Equations}.  [Graduate Studies in Mathematics], American Mathematical Society (1998).

\bibitem{eg} L. Evans, R.\ Gariepy. \textit{ Measure Theory and Fine Properties of Functions}.  [Studies in Advanced Mathematics], CRC Press.

\bibitem{fgn}  T. Franco, P. Gon\c calves, A. Neumann.
\textit{Hydrodynamical behavior of symmetric exclusion with slow bonds}. Accepted for publication in Annales de l'Institut Henri Poincar\'e: Probability and Statistics (B).

\bibitem{fl} T. Franco and C. Landim: \textit{Hydrodynamic Limit of Gradient Exclusion Processes with conductances}. Arch. Ration. Mech. Anal., 195, no. 2, 409--439, (2010).

\bibitem{l} G. Leoni. \textit{A First Course in Sobolev Spaces}.  [Graduate Studies in Mathematics], American Mathematical Society (2009).

\bibitem{L.}
T. Liggett. \textit{Interacting Particle Systems}. Springer-Verlag, New York (1985).

\bibitem{kl} C. Kipnis, C. Landim. \textit{Scaling limits of interacting
  particle systems}. Grundlehren der Mathematischen Wissenschaften
  [Fundamental Principles of Mathematical Sciences], 320.
  Springer-Verlag, Berlin (1999).

\bibitem{TE} R. Temam. \textit{Navier-Stokes Equations: Theory and Numerical Analysis}. Studies in mathematics and its applications ; v. 2.
North-Holland Pub. (1977).

\end{thebibliography}

\end{document}